\documentclass[absolute,12pt]{mymathart}
\usepackage{mymathmacros}

\usepackage{graphicx,color}

\usepackage{enumerate}

\numberwithin{equation}{section}

 \newtheoremstyle{numberedstyle}% name
   {9pt}%      Space above, empty = `usual value'
   {9pt}%      Space below
   {\normalfont}% Body font
   {}%         Indent amount (empty = no indent, \parindent = para indent)
   {\bfseries}% Thm head font
   {.}%        Punctuation after thm head
   {\newline}% Space after thm head: \newline = linebreak
   {}%         Thm head spec

\newcommand{\interior}{\operatorname{int}}
\newcommand{\hypdim}{\dim_{\operatorname{hyp}}}
\newcommand{\BlogP}{\mathcal{B}_{\log}^{\operatorname{p}}}

\newtheorem{thm}{Theorem}[section]%
\newtheorem{lem}[thm]{Lemma}%
\newtheorem{cor}[thm]{Corollary}%
\newtheorem{prop}[thm]{Proposition}%

\newcommand{\V}{\mathcal{V}}

\theoremstyle{numberedstyle}
\newtheorem{question}[thm]{Question}%
\newtheorem{defn}[thm]{Definition}%
\newtheorem{rem}[thm]{Remark}

\newcommand{\sign}{\operatorname{sign}}
\newcommand{\B}{\mathcal{B}}
\renewcommand{\S}{\mathcal{S}}
\newcommand{\T}{\mathcal{T}}
\newcommand{\HH}{\mathbb{H}}

\usepackage{hyperref}

\title[Hyperbolic entire functions with full hyperbolic dimension]{Hyperbolic entire functions \\ with full hyperbolic dimension \\ and approximation by
   Eremenko-Lyubich functions}
\author{Lasse Rempe-Gillen}
\thanks{This work was supported by EPSRC Fellowship EP/E052851/1.}

\newcommand{\dimh}{\dim_{\operatorname{H}}}

\newtheorem{standingassumption}[thm]{Standing Assumption}

\begin{document}

 \begin{abstract}
   We show that there exists a hyperbolic entire function $f$ of finite order
     of growth such that the hyperbolic dimension---that is, the
     Hausdorff dimension of the set of points in the Julia set of $f$ whose orbit
     is bounded---is equal to two. This is in contrast to the rational case, where
     the Julia set of a hyperbolic map must have Hausdorff dimension less than two,
     and to the case of all known explicit hyperbolic entire functions. 

 In order to obtain this example, we prove a general result on constructing
  entire functions in the 
  \emph{Eremenko-Lyubich class} $\B$ with prescribed behavior near infinity,
  using Cauchy integrals. This result significantly increases the class of functions
  that were previously known to be approximable in this manner. 

 Furthermore, we show that the approximating functions are quasiconformally
   conjugate to their original models, which simplifies the construction of 
   dynamical counterexamples. We also give some further applications of our
   results to transcendental dynamics.
 \end{abstract}

\maketitle

 \section{Introduction}
  The Hausdorff dimension
   $\dimh(J(f))$ of the Julia set of a rational function $f$ has been extensively studied.
   A related quantity, the \emph{hyperbolic dimension} $\hypdim(f)$, 
   was introduced by Shishikura \cite{mitsudim} as the
   supremum over the Hausdorff dimensions of hyperbolic subsets of $J(f)$.
   (Here a \emph{hyperbolic set} $K\subset J(f)$ is a compact, forward invariant subset
    of $J(f)$ such that sufficiently high iterates of $f$ are expanding when restricted
    to $K$.) 

  Clearly the hyperbolic dimension is a lower bound for $\dim(J(f))$. If the rational function
    $f$ is hyperbolic, then by definition $J(f)$ is a hyperbolic set itself, and hence
    \begin{equation} \label{eqn:dimequality}
      \dim J(f) = \hypdim(f). 
    \end{equation}
  It is natural to ask whether (\ref{eqn:dimequality}) holds more generally.
   In other words, how   
     prevalent is expanding dynamics in the Julia set of a rational function?
 \begin{question}\label{qu:rational}
  Is there a rational function $f$ such that $\dim J(f) \neq \hypdim(f)$? 
 \end{question}
 
 The relation (\ref{eqn:dimequality}) is known to hold in a vast number of cases,
   including all non-recurrent rational functions. 
  The tantalizing possibility of an example of a rational function where
   (\ref{eqn:dimequality}) fails was first suggested by results
  of Avila and Lyubich \cite{avilalyubichfeigenbaum2} on Feigenbaum quadratic
  polynomials with periodic combinatorics. They show that, \emph{if} such a map
  exists whose Julia set has positive area (and hence dimension $2$), 
  then its hyperbolic dimension 
  would need to be strictly less
  than two. Since this article was submitted for publication, 
  Avila and Lyubich have announced a proof
  that such Feigenbaum Julia sets of positive area do indeed exist, answering Question  
  \ref{qu:rational} in the positive. It remains open whether
   there is a rational function $f$
   such that $\hypdim(f) < \dim J(f) < 2$.

 The results of Avila and Lyubich resonate
  strongly with the iteration theory of transcendental entire functions.
  Here, in stark contrast to the rational case, even \emph{hyperbolic} functions 
  (see Definition \ref{defn:disjointtype}) frequently
  satisfy $\dim(J(f))=2$ and $\hypdim(f)<2$. 
  Stallard \cite{stallardhyperbolicmero} was the first to
  construct hyperbolic examples with $\hypdim(f)<\dim(J(f))$ (in slightly different terminology), 
  while Urba\'nski and Zdunik \cite{urbanskizdunik1} proved
  that this situation occurs for hyperbolic exponential maps $f(z)=\exp(z)+a$, where
  $\dim(J(f))=2$ by a result of McMullen \cite{hausdorffmcmullen}. This
  suggests that a systematic
  understanding of the measurable dynamics of transcendental
  entire functions is not only interesting in its own right, but can also help to
  shed further light on phenomena such as those discovered by Avila and Lyubich.

  A class of hyperbolic entire functions that has received particular attention
  in recent years is given by those of finite order and disjoint type:
 \begin{defn} \label{defn:disjointtype}
   A transcendental entire function $f:\C\to\C$ is \emph{hyperbolic} if
    there exists a compact set $K\subset\C$ with $f(K)\subset \interior(K)$ such that
    the restriction
     \[ f:f^{-1}(\C\setminus K)\to \C\setminus K \]
    is a covering map.
   An entire function is said to be \emph{of disjoint type} if it is hyperbolic and the Fatou set $F(f)$ is 
    connected, or equivalently if the set $K$ can be chosen to be connected.

  An entire function has \emph{finite order} if, setting
    $\log_+r=\max(0,\log r)$, we have
     \[ \limsup_{z\to\infty} \frac{ \log_+\log_+ |f(z)|}{\log_+|z|} < \infty. \]
 \end{defn}

 The topology of the Julia set of a hyperbolic entire function of finite order 
  is completely understood \cite{baranskihyperbolic,brush,boettcher,strahlen}.
  Moreover, the Hausdorff dimension of $J(f)$ is equal to two in this case
  \cite{baranskihausdorff}, and the hyperbolic dimension is greater than one
  \cite{baranskikarpinskazdunik}. 
  More precisely, suppose that $f$ is of disjoint type and finite
  order. Then:
  \begin{itemize}
    \item The Julia set is a disjoint uncountable union of curves to $\infty$,
      each consisting of a finite \emph{endpoint} and a \emph{ray} connecting this 
      endpoint to infinity. In fact, $J(f)$ is ambiently homeomorphic to a
      straight brush in the sense of \cite{aartsoversteegen} (i.e., a certain universal
      plane topological object). 
     \item The set of endpoints in $J(f)$ has Hausdorff dimension equal to two.
     \item The union of rays in $J(f)$ (without endpoints) 
        has Hausdorff dimension equal to one.
  \end{itemize}

  Furthermore, for a large class of hyperbolic entire functions of finite order,
   including all hyperbolic maps in the exponential family $z\mapsto \exp(z)+a$, the    
   trigonometric family $z\mapsto a\exp(z) + b\exp(-z)$ and many others,
   the measurable dynamics is described in detail by the results of
   \cite{mayerurbanskifractal,mayerurbanski}. In particular, these maps satisfy
   $\hypdim(f)<2=\dim(J(f))$. This suggests the following problem.

 \begin{question}
  Let $f:\C\to\C$ be a hyperbolic transcendental entire function of finite order.
   Is it always the case that $\hypdim(f)<2$?
\end{question}

 In this article, we give a negative answer.

 \begin{thm}[(Hyperbolic functions with full hyperbolic dimension)]\label{thm:hypdim}
  There exists a transcendental entire function $f$ of disjoint type and finite order
    such that $\hypdim(f)=2$.
   Furthermore, $f$ can be chosen such that $J(f)$ has positive measure. 
 \end{thm}
 
\subsection*{Approximation} In order to obtain the desired counterexample,
  we shall first construct a suitable ``model function'' with the desired behavior, and
   then approximate this map by an entire function. The idea of using
   approximation to construct interesting examples in complex dynamics
   was introduced by Eremenko and Lyubich \cite{alexmishaexamples}, who used
   Arakelyan's theorem \cite[Satz IV.2.3]{gaier}, an important result of approximation
   theory. Given a closed set $A$, this theorem states that any function $g$, 
   defined and continuous on
   $A$ and holomorphic on its interior, can be uniformly approximated by
   entire functions if and only if $A$ satisfies certain simple topological
   conditions. 

  Arakelyan's theorem, while powerful, has the drawback that there is little we can say about
  the behavior of the approximating function $f$ outside of the set $A$. 
  In particular, we have no control over 
  the set $\sing(f^{-1})$ of critical and asymptotic values
  of $f$, which prevents us from being able to restrict the global function-theoretic
 or dynamical properties. Indeed, in order to obtain hyperbolic examples,
  we will at least need to be sure that the approximating function belongs to
  the \emph{Eremenko-Lyubich class}
    \[ \B := \{f:\C\to\C
  \text{ transcendental entire: $\sing(f^{-1})$ is bounded}\},\]
   which was introduced in \cite{alexmisha}. A related
   problem, which we 
   mention here for completeness although it is not treated in this
   article, is approximation by functions in the 
   \emph{Speiser class}
    \[ \S := \{f:\C\to\C\text{ transcendental entire: $\sing(f^{-1})$ is finite}\}
         \subset \B. \]

  Leaving aside dynamics for a moment, let us discuss the question  
   of the function-theoretic
   behavior that these maps can exhibit. As it turns out, this will be the main 
   problem to deal with when constructing hyperbolic examples. 

  If $f\in\B$, then for sufficiently large $R>0$, every component of
   $f^{-1}(\{|z|>R\})$ is simply connected and mapped by $f$ as a universal
   covering. These components are called the 
   \emph{tracts} of $f$ (over $\infty$). 
  If $f\in\B$ and $T$ is a tract of $f$ as above, then we can define
   a branch of $\log f$ on $T$, and 
   $\log f - \log R:T\to \HH$ is a conformal isomorphism
   (where $\HH := \{\re z > 0\}$ denotes the right half plane).
  Conversely, it is natural to ask which universal coverings can be
  approximated by entire functions in the class $\B$. 

 \begin{question} \label{qu:approx}
  Suppose that $T\subset\C$ is a Jordan domain whose boundary passes through infinity
   and that $\Psi:T\to\HH$ is a conformal isomorphism with $f(\infty)=\infty$. 
   Under which conditions on
   $\Psi$ does there exist an entire
   function $f\in\B$ (resp.\ $f\in\S$) such that
    \[ f(z) = e^{\Psi(z)} + O(1), \quad z\in T\ ?\] 
 \end{question}

  Arakelyan's theorem implies that such a function $f$ always exists if we
   drop the requirement that $f\in\B$. 

  A well-known way of building functions with a given tract is to
   use Cauchy integrals; see e.g.\ \cite[Part III, problem 158]{polyaszego}. 
   Although this method
   is rather old, and has had many applications over the years, it does not
   seem to have been treated systematically in the classical literature. 
   The only theorem of a general nature that we are aware of was recently
   stated in \cite{strahlen}, following a construction from a paper
   by Eremenko and Gol'dberg \cite{eremenkogoldberg}. The result in 
   \cite[Proposition 7.1]{strahlen} states that approximation
   is always possible when $\Psi$ is the restriction of a conformal isomorphism
   $\Psi:T'\to \Sigma$, where $\Sigma$ is a sector,
   $\Sigma=\{z\in\C: |\arg z | < \pi/2 + \eps\}$.

  In \cite{strahlen}, this general theorem is used to construct a 
   counterexample to the so-called strong Eremenko conjecture. However, the
   requirement that $\Psi$ extends to a conformal isomorphism onto a sector
   of opening angle
   greater than $\pi$ is rather strong. It prevents, for instance, the construction
   of functions of lower order $1/2$, as well as of tracts  
   such as the one depicted in \cite[Figure 1]{devaneyhairs}. 

  In this note, we present a considerable
   strenghtening of \cite[Proposition 7.1]{strahlen}, which states that
   approximation is always possible if $\Psi$ extends to a conformal isomorphism whose
   domain 
   is only ``slightly'' larger than the half plane $\HH$. 
   It is convenient to first introduce the following definition. 

   \begin{defn}[(Model functions)] \label{defn:modelfunction}
   A \emph{model function} is a conformal isomorphism
     \[ \Psi: T \to H, \]
   where 
    \begin{itemize}
     \item $T\subset\C$ is an unbounded  simply-connected domain;
     \item $H$ is a simply-connected domain with $\HH\subset H$
      (recall that 
       $\HH := \{\re z > 0\}$ denotes the right half plane);
     \item if $z_n\in T$ is a sequence with $f(z_n)\to\infty$ in $H$, then
            $z_n\to\infty$ in $T$. 
    \end{itemize}   
  \end{defn}   

  \begin{thm}[(Approximation of model functions)]\label{thm:approx}
   Let 
    \[ H := \{x+iy: x>-14 \log_+|y|\}, \]
   where $\log_+(t) := \max(0,\log t)$, and let
   $\Psi:T\to H$ be a model function.

   Set
     $g := \exp\circ\Psi$. Then there exists an entire function
    $f\in\B$ such that    
      \begin{align*}
        f(z) &= g(z) + O\left(\frac{1}{z}\right)\quad
    \text{when $z\in T$,} \quad \text{and} \\
      f(z) &= O\left(\frac{1}{z}\right)
   \quad \text{when $z\notin T$} \end{align*}
   (as $z\to \infty$).
%%% Move these to a more precise statement later on??
%
%   Moreover, set $z_0 := \Psi^{-1}(1)$. Then
%    the constants in the $O$-notation depend only
%    on (an upper bound for) $|z_0|$ and (a lower bound for)
%    $\dist(z_0,\partial T)$. 
%
   If the domain $T$ is symmetric with respect to the real axis and
    $\phi(T\cap\R)\subset\R$, then $f$ can be chosen such that
    $f(\R)\subset\R$.
  \end{thm}

\subsection*{Dynamical approximation} 
 In order to use Theorem \ref{thm:approx} to prove Theorem
  \ref{thm:hypdim}, we observe that the 
  approximation automatically preserves dynamical features.
  The key fact is that, given our quality of approximation, the 
  functions $f$ and $g$ are \emph{quasiconformally
  equivalent near $\infty$} in the sense of \cite{boettcher}:

 \begin{thm}[(Quasiconformal equivalence)] \label{thm:equivalence}
  Let $f$ and $g$ be as in Theorem \ref{thm:approx}, and let $R>0$ be
   sufficiently large.

  Then there exists a quasiconformal homeomorphism $\phi:\C\to\C$ such that
    \[ g(z) = f(\phi(z)) \]
  for all $z\in \C$ with $|g(z)|\geq R$. 

  This map is asymptotically conformal at $\infty$; more precisely,
    \[ \phi(z) = z + O(1) \]
  as $z\to\infty$. 
 \end{thm}

 By \cite{boettcher}, this implies that the functions are quasiconformally
  conjugate on the set of points whose orbits
  stay suitably large under iteration. In our setting 
  we can even be sure (adapting ideas from \cite{boettcher}) 
  to obtain a \emph{global} conjugacy on the Julia sets of the two functions,
  provided that  
  the tract $T$ is sufficiently well inside the domain $\{|z|>1\}$. 

 \begin{thm}[(Quasiconformal conjugacy)]\label{thm:conjugacy}
  There is a universal constant $\rho_0>1$ with the following
   property.

  Let $\Psi:T\to H$ be as in
   Theorem \ref{thm:approx}, with the additional property that
   $T\subset \{|z|>\rho_0\}$.

  Then, again setting $g:=\exp\circ\Psi$,
   the function $f$ in Theorem 
   \ref{thm:approx} can be chosen 
   such that $f$ and $g$ are quasiconformally conjugate
   near their Julia sets. 

  More precisely, there is a quasiconformal
   homeomorphism $\theta:\C\to\C$ such that 
     \[ \theta(g(z)) = f(\theta(z)) \]
    whenever $|g(z)|\geq \rho_0$, and $\theta$ restricts to a homeomorphism
    between the Julia set $J(g)$
    (i.e., the set of points that remain in $T$ under iteration of $g$)
    and the Julia set of $f$. 
    Furthermore,
    the complex dilatation of $\theta$ equals zero almost everywhere on
    $J(g)$. 
 \end{thm}

In order to prove Theorem \ref{thm:hypdim}, it will thus be sufficient
  to construct
  a model function $\Psi:T\to H$ for which the map $g := \exp\circ\Psi$ contains
  hyperbolic subsets of dimension arbitrarily close to $2$ and which satisfies
  the hypotheses of Theorem \ref{thm:conjugacy}. (Recall that quasiconformal
  mappings preserve sets of Hausdorff dimension $2$.)

\subsection*{Further applications}
  Let us state two further new theorems that can be obtained
  from our approximation results via known constructions. (We 
  refer to the articles in question for background on the
  questions answered by these examples.) The first
  is a strengthening of \cite[Theorems 8.2 and 8.3]{strahlen}.

 \begin{thm}[(Counterexamples of low growth to the strong Eremenko conjecture)]
   \label{thm:eremenkolowgrowth}
  There exists a disjoint-type transcendental entire function $f\in\B$ such that
  \begin{enumerate}[(a)]
    \item $\displaystyle{\log_+ \log_+ |f(z)| = (\log_+|z|)^{1+o(1)}}$ as 
         $z\to\infty$,
   \item $f$ has lower order $1/2$, and
   \item the Julia set $J(f)$ has no unbounded path-connected components.
      \label{item:eremenko}
  \end{enumerate}   
 \end{thm}
 \begin{remark}[Remark 1]
  We recall that any function $f\in\B$ must have lower order at least $1/2$
   and that no function $f\in\B$ of finite (upper) order can satisfy
   (\ref{item:eremenko}) by \cite{strahlen}. 
 \end{remark}
 \begin{remark}[Remark 2]
  The condition on the growth of $f$ implies that
   $J(f)$ has Hausdorff dimension equal to two 
   \cite{bergweilerkarpinskastallard}.
 \end{remark}

 Our second application strengthens a counterexample from 
   \cite{devaneyhairs}.

 \begin{thm}[(Slowly escaping Devaney hairs)] \label{thm:devaney}
  There exists a disjoint-type transcendental entire function $f\in\B$ such that
  \begin{enumerate}[(a)]
   \item $f(\R)\subset\R$ and $J(f)\cap \R=[a,\infty)$ for some $a>0$;
   \item every connected
    component of $J(f)$ is an arc connecting some finite
    endpoint to infinity, and $[a,\infty)$ is such a component;
   \item every $x>a$ belongs to the \emph{escaping set}
     $I(f) = \{z\in\C: f^n(z)\to\infty\}$;
   \item the real axis does not intersect the \emph{fast escaping set}
     $A(f)$ of points that escape to infinity ``as fast as possible'' in 
     the sense of Bergweiler and Hinkkanen
     \cite{bergweilerhinkkanen}. 
  \end{enumerate}
 \end{thm} 

 We also note that the same construction as in the proof of Theorem
   \ref{thm:hypdim}, with different parameters, suggests a counterexample
   to the \emph{area conjecture} of Epstein and Eremenko. However, while this construction
   yields a counterexample in the class of model functions, our approximation result does not
   allow us to construct such a counterexample in the class $\B$.
   This application and its background is discussed in \cite{area}.

\subsection*{Some remarks about the proof}
 As already mentioned, the proof of Theorem \ref{thm:approx} uses
  Cauchy integrals. More precisely, let $\gamma:(-\infty,\infty)\to T$
  be defined by
    \[ \gamma(t) := \Psi^{-1}(it -13\log_+|t| + 1 ) \]
  and consider the function
    \[ h(z) := \frac{1}{2\pi i}\int_{\gamma} \frac{g(\zeta)}{\zeta-z}d\zeta. \]
  We will show that the integral converges absolutely for
   $z\notin \gamma$ and that
    \[ h(z) = O\left(\frac{1}{z}\right) \]
   as $z\to\infty$. It then follows that 
    \[ f(z) := \begin{cases}
                h(z)+g(z) & z\in\tilde{T} \\
                 h(z) & z\notin \tilde{T} \end{cases} \]
   is the desired entire function, where
   $\tilde{T}$ is the component of $\C\setminus \gamma$ that is contained
   in $T$. 

 We should comment that the constant $14$ appearing in the definition
  of $H$ is not best possible: Our proof shows
  that it can be replaced by any constant that is larger than $13$, and
  with some more careful estimates, it
  could be reduced further. However, our proof does not yield 
  the analog of Theorem \ref{thm:approx}
  for a domain of the form
    \[ H := \{x+iy: x>-\eps \log_+|y|\}, \]
   where $\eps>0$ is arbitrarily small.

  We also note that, for the application in Theorem \ref{thm:hypdim}, it is important that the domain $H$ is of the
   form as above (compare Remark 1 after the proof of Theorem
    \ref{thm:hypdiminproof}): e.g.\ it would not be sufficient to be able to approximate functions
   $\Psi:T\to H$, where
     \[ H = \{x + iy: x > -|y|^{\eps}\}. \]

\subsection*{Subsequent results}
Motivated by our results, Chris Bishop \cite{bishopclassB} has recently given a complete answer to Question \ref{qu:approx}, if one considers
  quasiconformal equivalence instead of uniform approximation: If $\Psi:T\to H$, then there is a function $f\in\B$, with a single tract,
  that is quasiconformally
  equivalent to $\exp\circ\Psi$ near $\infty$. Furthermore, there is a function $f\in\S$ such that a suitable restriction of $f$ is quasiconformally equivalent to
  $\exp\circ\Psi$. (In general, $f$ cannot be constructed to have only a single tract.) These results even hold for arbitrary unions of tracts
  that accumulate only at infinity. In particular, Bishop's methods allow the construction 
  of a counterexample to the area conjecture mentioned above, 
  in the class $\S$, and indeed a counterexample \cite{bishoporder} to the stronger
  \emph{order conjecture} of Adam Epstein, which asked whether the order of a transcendental 
  entire function $f\in\S$ is invariant under quasiconformal equivalence. 

\subsection*{Structure of the article}
  The first part of the paper deals with approximation and the proof of Theorem
   \ref{thm:approx}.
   In Section \ref{sec:cauchy}, we prove a technical result about the 
   approximation of holomorphic functions using Cauchy integrals. (This 
   covers a number of known constructions.) In Section
   \ref{sec:hyperbolicmetric}, we collect some basic facts about hyperbolic
   geometry in plane domains; these are used in Section \ref{sec:approx} to
   prove Theorem \ref{thm:approx} in a slightly more general framework, using
   the results from Section \ref{sec:cauchy}. The
   short Section \ref{sec:initial} is dedicated to verifying that our 
   hypotheses in Theorem \ref{thm:approx} indeed satisfy the assumptions
   used in Section \ref{sec:approx}.

 The second part of the paper consists of Section \ref{sec:conjugacy}, which
   establishes the results on quasiconformal equivalence and conjugacy.

 Finally, Section \ref{sec:hypdim} constructs the model function required
   for the proof of Theorem \ref{thm:hypdim}, while Section
   \ref{sec:examples} briefly discusses Theorems \ref{thm:eremenkolowgrowth} and
   \ref{thm:devaney}.

  We remark that the three parts of the paper can be read quite independently
    of each other (with the exception that the hyperbolic metric estimates of Section
    \ref{sec:hyperbolicmetric} will be used throughout). 

\subsection*{Acknowledgments}
  I owe great thanks to Alexandre Eremenko, who introduced me to the method
   of approximation via Cauchy integrals by pointing me to the paper
   \cite{eremenkogoldberg}, and who has shared many profound insights on this
   and related problems. I would also like to thank Adam Epstein, 
   who led me to think about the area conjecture and
   to discover the basic structure of the example in
   Theorem \ref{thm:hypdim}, and Peter Hazard, stimulating conversations with 
   whom resulted in the realization that this example could be adapted to yield
   functions with full hyperbolic dimension. Finally, I would like to thank Chris Bishop, 
   Helena Mihaljevi\'c-Brandt, Phil Rippon, Gwyneth Stallard and Mariusz
   Urba\'nski for interesting discussions about this work.

\subsection*{Basic notation}
 As usual, we denote by $\C$ the complex plane. We also denote the
  right half plane by 
   \[ \HH := \{a+ib: a>0, b\in\R\} \]
  and the (Euclidean)
  disk of  radius $r$ around a point $z_0\in\C$ by
   \[ \D_r(z_0) := \{z\in\C: |z-z_0|<r\}. \]
  Euclidean distance is denoted $\dist$; e.g.\ $\dist(A,z_0)$ is the Euclidean distance between a set $A\subset\C$ and the point $z_0$.  

 As mentioned above, we set $\log_+(t) := \max(0,\log(t))$ for
   $t\geq 0$. We also define
   \[ |z|_+ := \max(|z|,1) = \exp(\log_+|z|) \]
  for all $z\in\C$.

 If $f:\C\to\C$ is a transcendental entire function, we denote by
  $\sing(f^{-1})$ its set of critical and asymptotic values.
  (Here $a$ is an asymptotic value if there is a curve $\gamma:[0,\infty)\to \C$ with
   $\gamma(t)\to\infty$ and $f(\gamma(t))\to a$ for $t\to\infty$.) The closure
   of $\sing(f^{-1})$ (in $\C$) is denoted $S(f) := \cl{\sing(f^{-1})}$. An alternative
   definition of $S(f)$, which is the one we will be using, is as the
   smallest closed set that has the property that
    \[ f:\C\setminus f^{-1}(S(f)) \to \C\setminus S(f) \]
   is a covering map.

\section{Approximation using Cauchy integrals}\label{sec:cauchy}

 In this section, we prove a general technical
  result about the approximation of
  holomorphic functions by Cauchy integrals. In Section \ref{sec:approx},
  this will be used
  to deduce our main approximation theorem (Theorem \ref{thm:approx}).

\begin{thm}[(Convergence of Cauchy integrals)] \label{thm:cauchy}
  Let $T\subset\C$ be a simply-connected domain and let
   $g:T\to\C$ be holomorphic. Let
   $\gamma:(-\infty,\infty)\to T$ be an injective and piecewise smooth
   curve such that $|\gamma(t)|\to\infty$ as $|t|\to\infty$, and let
   $\tilde{T}\subset T$ be the component of $\C\setminus\gamma$ that is contained 
   in $T$. We assume that $\gamma$ runs around $\tilde{T}$ in clockwise direction.

  Suppose furthermore that there are constants $C_1, \dots, C_5\geq 1$ 
   and $\delta_1,\delta_2\geq 0$ such that the following hold for 
   all $\tau\in\R$ (recall that $|\tau|_+=\max(|\tau,1|)$): 
  \begin{enumerate}[(a)]
   \item $|\gamma(\tau)|\leq C_1\cdot |\tau|_+$,
     \label{item:gamma}
   \item $|\gamma'(\tau)|\leq C_2\cdot |\tau|_+^{\delta_1}$,
     \label{item:gamma'}
   \item $|g(\gamma(\tau))|\leq C_3\cdot |\tau|_+^{-(2+\delta_1+\delta_2)}$, and
     \label{item:g}
   \item if $|z-\gamma(\tau)|\leq |\tau|_+^{-\delta_2}/C_4$, then
       $z\in T$ and $|g(z)|\leq C_5\cdot |\tau|_+^{-1}$. 
     \label{item:disk}
  \end{enumerate}

 Then 
    \begin{equation} \label{eqn:defnh}
     h(z) := \frac{1}{2\pi i}\int_{\gamma} \frac{g(\zeta)}{\zeta-z}d\zeta 
    \end{equation}
   defines a holomorphic function for $z\notin \gamma$, and 
    \[ f(z) := \begin{cases}
                h(z)+g(z) & z\in\tilde{T} \\
                 h(z) & z\notin \tilde{T} \end{cases} \]
   extends to an entire function $f:\C\to\C$. 

  Furthermore, there is a constant $C_6$ such that
    \[ |h(z)| \leq \frac{C_6}{|z|_+}, \]
   where $C_6$ depends only on $C_1,\dots,C_5$; more precisely,
   $C_6=O(C_1\cdot C_2\cdot C_3\cdot C_4\cdot C_5)$. 
\end{thm}
\begin{proof}
  We have
    \[ |g(\gamma(\tau))|\cdot |\gamma'(\tau)| \leq
        C_2\cdot C_3 \cdot |\tau|_+^{-(2+\delta_2)}, \]
   hence 
    \begin{align} \label{eqn:absolute}
     \int_{\gamma} |g(\zeta)|\,|d\zeta| &=
       \int_{-\infty}^{\infty} |g(\gamma(\tau))|\cdot |\gamma'(\tau)|\, d\tau  \\
     &\leq C_2\cdot C_3\cdot \int_{-\infty}^{\infty}|\tau|_+^{-(2+\delta_2)}\,d\tau =
        2\cdot C_2\cdot C_3\cdot \bigl(1+\frac{1}{1+\delta_2}\bigr) \leq
       4\cdot C_2\cdot C_3. \notag \end{align}
   This implies that the integral in (\ref{eqn:defnh}) is absolutely convergent and defines a holomorphic
    function $h$ on $\C\setminus \gamma$. 
    If $z_0\in\gamma$, then we can
    modify the curve $\gamma$ slightly to avoid the point $z_0$, and 
    thus
    see that the restriction of $h$ to $\tilde{T}$ has
    an analytic extension to a neighborhood of $z_0$; the same is 
    true for the restriction $h|_{\C\setminus\cl{\tilde{T}}}$.
    Using the residue theorem, we see that the
    two extensions differ exactly by the function
    $g(z)$ in a neighborhood of $z_0$, which shows that
    the function $f$ defined in the statement of the theorem
    does indeed extend to an entire function. (Compare also 
    \cite[Claim 2 in Section 7]{strahlen}.)

   Thus it remains to prove that $|h(z)|=O(1/|z|_+)$. The main problem is
    to estimate $h(z)$ when $z$ is close to some point $\gamma(\tau_0)$. In this
    case, we will modify $\gamma$ to a curve $\gamma^z$ that avoids the disk
    $D_{\tau_0}$ of radius $\delta(\tau_0) := |\tau_0|_+^{-\delta_2}/C_4$ around $\gamma(\tau_0)$.

   More precisely, let $z\in\C\setminus\gamma$.
    If $|\gamma(\tau)-z|>\delta(\tau)/2$ for all $\tau\in\R$, then
    we set $\gamma^z := \gamma$. 
    Otherwise
    choose $\tau_0$ with 
    $|\gamma(\tau_0)-z|\leq \delta(\tau_0)/2$ such that $|\tau_0|$ is 
    minimal. Let $\tau_1$ and $\tau_2$ be
    the smallest, respectively largest, values of $\tau$ for which
    $\gamma(\tau)\in\partial D_{\tau_0}$, and set
     \[ \gamma^z := \gamma\bigl((-\infty,\tau_1)\bigr) \cup
           \alpha \cup \gamma\bigl([\tau_2,\infty)\bigr), \]
    where $\alpha$ is an arc of 
    $\partial D_{\tau_0}$ chosen such that 
     $\gamma^z$ is homotopic to $\gamma$ in $C\setminus\{z\}$. 

   We then have 
\[       h(z) = \frac{1}{2\pi i}\int_{\gamma^z} \frac{g(\zeta)}{\zeta-z}\,d\zeta.
  \]
   Hence
    \begin{equation} \label{eqn:splitting1}
      2\pi |h(z)| \leq 
       \int_{\partial D_{\tau_0}}\frac{|g(\zeta)|}{|\zeta-z|}|d\zeta| +
       \int_{\gamma\setminus \cl{D_{\tau_0}}}\frac{|g(\zeta)|}{|\zeta-z|}|d\zeta|. 
    \end{equation}

\smallskip

  To estimate the first integral,
    we bound $|\tau_0|_+$ from below in terms of $|z|_+$. We have 
    \begin{equation}
     |\gamma(\tau_0) - z | \leq \frac{\delta(\tau_0)}{2} =
      \frac{|\tau_0|_+^{-\delta_2}}{2C_4}\leq 1.
    \end{equation}
   Thus $|z|_+\leq 2|\gamma(\tau_0)|_+$, and hence, by 
    (\ref{item:gamma}), 
      \begin{equation} \label{eqn:sizeoft0}
        |\tau_0|_+ \geq \frac{|\gamma(\tau_0)|_+}{C_1}\geq
          \frac{|z|_+}{2C_1}.
      \end{equation}

  So, by choice of $D_{\tau_0}$ and (\ref{item:disk}), we can 
    bound the first integral from (\ref{eqn:splitting1}):
     \begin{equation}
        \label{eqn:estimate1} 
        \int_{\partial D_{\tau_0}} \frac{|g(\zeta)|\,|d\zeta|}{|\zeta-z|} \leq 
        \int_{\partial D_{\tau_0}} \frac{C_5\,|d\zeta|}{|\tau_0|_+\delta(\tau_0)}
        = \frac{2\pi C_5}{|\tau_0|} 
       \leq \frac{4\pi C_1 C_5}{|z|_+}. 
      \end{equation}
 \smallskip

  Now we turn to 
   estimating the second integral in
   (\ref{eqn:splitting1}).
   If $\zeta=\gamma(\tau)\in \gamma^z\setminus \cl{D_{\tau_0}}$, then we have 
   $|\zeta - z|>\delta(\tau_0)/2$. Using the definition of $\tau_0$ and
    monotonicity of the function $|\tau| \mapsto \delta(\tau)$, we see that 
     \begin{equation}\label{eqn:distancegammaz}
      |\gamma(\tau)-z|> \frac{\delta(\tau)}{2} = 
           \frac{1}{2 C_4|\tau|_+^{\delta_2}}. 
     \end{equation}
    This estimate, together
     with (\ref{eqn:absolute}), 
     would be sufficient to prove that the integral in question,
     and hence $h(z)$,  is bounded. 
    In
    order to obtain the stronger fact that $h(z)=O(1/z)$, we subdivide
    the remaining part of the curve once more. (We note that
    this stronger bound is not required for the applications
    that we have in mind.) 

  Define  
   $\Theta := \max\left(1,\frac{|z|_+}{2C_1}\right)$. For
   $\tau\leq \Theta$, we then have  
    \[ |\gamma(\tau)|\leq C_1\cdot |\tau|_+ \leq C_1\cdot\Theta =   \max\left(\frac{|z|_+}{2},C_1\right)=:R. \]
   We use this to estimate the integral over the curve
     \[ \gamma^z_1 := \gamma^z\cap \gamma\bigl([-\Theta,\Theta]\bigr). \]
   The idea is that $|\zeta-z|$ is (at least) comparable to $|z|_+$ for
    all points $\zeta$ on this curve. Indeed, suppose that
    $|z|_+<2C_1$. Then $\Theta=1$, and hence, by     (\ref{eqn:distancegammaz}), 
     \begin{equation} \label{eqn:zetazestimate}
       |\zeta - z | \geq \frac{1}{2C_4\Theta^{\delta_2}} =
            \frac{1}{2C_4} >  \frac{|z|_+}{4C_1 C_4}. 
     \end{equation} 
If $|z|_+\geq 2C_1$, then $|\zeta|\leq R = |z|_+/2$, hence again
      $|\zeta-z|\geq |z|_+/2>|z|_+/(4C_1 C_4)$.  Thus, 
    using (\ref{eqn:absolute}): 
     \begin{align} \label{eqn:estimate2}
        \int_{\gamma^z_1} \frac{|g(\zeta)|}{|\zeta-z|}|d\zeta| &\leq
        \frac{4C_1 C_4 }{|z|_+}\int_{\gamma^z_1}|g(\zeta)|\,|d\zeta| \\
     &\leq \notag
        \frac{4C_1 C_4}{|z|_+}\int_{\gamma}|g(\zeta)|\,|d\zeta| \leq
       \frac{16C_1 C_2 C_3 C_4}{|z|_+}.
     \end{align}

 It remains to deal with the part of the curve given by 
 \[
          \gamma^z_2 := 
                 \gamma\setminus     
        \left(\cl{D_{\tau_0}}\cup \gamma\bigl([-\Theta,\Theta]\bigr)\right). \]
   We use (\ref{eqn:distancegammaz}),
    as well as (\ref{item:gamma'}) and (\ref{item:g}) to obtain 
      \begin{align}\label{eqn:estimate3}
    \int_{\gamma^z_2} \frac{|g(\zeta)|}{|\zeta-z|}|d\zeta| &\leq
         \int_{|\tau|>\Theta} 2C_4 |\tau|^{\delta_2}\cdot |g(\gamma(\tau))|
                           \cdot |\gamma'(\tau)|\,|d\tau| \\
       &\leq 4\cdot C_2\cdot C_3\cdot C_4\cdot \notag
         \int_{\Theta}^{\infty}
              t^{\delta_2}\cdot t^{-(2+\delta_1+\delta_2)}\cdot t^{\delta_1}\,dt \\ 
       &= \notag
        4\cdot C_2\cdot C_3\cdot C_4 \cdot 
         \int_{\Theta}^{\infty}
              \tau^{-2}\,d\tau 
      = 4C_2C_3 C_4 \Theta^{-1} \leq 
        \frac{8 C_1 C_2C_3 C_4}{|z|_+}.
    \end{align}

    Combining the estimates (\ref{eqn:estimate1}), (\ref{eqn:estimate2}) and
      (\ref{eqn:estimate3}), the proof is complete. 
\end{proof}

 The following proposition shows that any function approximating a universal
  covering must itself have a logarithmic singularity over infinity. 

 \begin{prop} \label{prop:logtract}
  Let $\Psi:T\to\HH$ be a model function, 
   and set $g := \exp\circ\Psi$. Suppose that
   $f:T\to\C\setminus\{0\}$ 
   is a holomorphic function with $|f(z)-g(z)|\leq M$ for some $M>0$ and all
   $z\in T$. 
   Define 
    \[ T' := \{z\in T: |f(z)|>2M\}. \]
   Then $T'$ is a simply-connected domain and $f:T'\to \{|z|>2M\}$ is a universal
   covering map. 
 \end{prop}
 \begin{proof}
  Let us set $T'' := \{z\in T: |g(z)|>M\}$. Then $T''$ is simply-connected,
  $|f(z)|>0$ for all $z\in T''$ and $T'\subset T''$. 
  It follows from the minimum principle that $T'$ is simply-connected. We can
  define
   a branch $F:T'\to \HH$ of $\log f$. By continuity, we have $\re F(z)\to \log(2M)$ as
   $z$ tends to a point in the boundary of $T'$ (in $\C$). We claim that
   $|F(z)|\to\infty$ as $z\to\infty$. Indeed, by assumption we have, for
   all $z\in T'$, 
     \[ |\re F(z) - \re \Psi(z)| =
        \left| \log\frac{|f(z)|}{|g(z)|}\right| \leq
           \log 2. \] 
   Furthermore, the argument of $f(z)$ and $g(z)$ differs by less than
   $\pi$, and hence 
    $|\im F(z) - \im \Psi(z)|$ is contained in the union
    \[ \bigcup_{k\in\Z} \bigl ((2k-1)\pi , (2k+1)\pi\bigr ). \]
   Since $T'$ is connected, it follows that
      \[ |\im F(z) - \im \Psi(z)| \leq (2k+1)\pi \]
    for some $k\in\Z$. Hence 
     \[ |F(z) - \Psi(z) | \leq K \] 
   for a suitable constant $K>0$, which proves our claim that $|F(z)|\to\infty$ as $z\to\infty$.  

  So $F$ is a proper map, and hence has some well-defined degree $d$. $F$ extends to
   a degree $d$ map from the boundary of $T'$ (in the Riemann sphere) to 
   $\{\re z = \log(2M)\}\cup\{\infty\}$. Since $\infty$ only has one preimage,
   it follows that $d=1$. Thus $F$ is a conformal isomorphism, and
   $f = \exp\circ F$ is a universal covering map, as claimed. 
 \end{proof}

Before proving our main approximation result in Section \ref{sec:approx}, 
  let us note that Theorem
  \ref{thm:cauchy} includes the examples from \cite{polyaszego} and 
  \cite{stallardentirehausdorff3}. 
   
\begin{cor}
 Let $p>0$ and set
   \begin{align*}
      S_1 &:= \{ x + iy: x>0\text{ and } |y|<\pi \}, \quad\text{and} \\ 
      S_2 &:= \{ x + iy: |y| \leq \pi x/[(1+p)(\log (x))^p], \; x \geq 3 \}. 
   \end{align*}
   Also set
    \[ g_1(z) := e^{e^{z}}\quad\text{and}\quad g_2(z) :=\exp(e^{(\log z)^{1+p}}) \]
    (where $g_1$ is defined on $\C$, and $g_2$ on $\C\setminus(-\infty,1]$). 
   Let $\gamma_j$ be the boundary of $S_j$, described in clockwise direction. Then 
    \[ f_j(z) := \frac{1}{2\pi i} \int_{\gamma} \frac{g_j(\zeta)}{\zeta-z}d\zeta,\quad
      \notin \cl{S}, \]
  extends to an entire function $f_j:\C\to\C$. Furthermore, $f_j\in\B$ and 
    \[ f_j(z)=\begin{cases}
              g_j(z) + O(1/z) & z\in S_j \\
              O(1/z) & \text{otherwise}\end{cases} \qquad\text{as $z\to\infty$.}\]
\end{cor}
\begin{proof} 
 It is easy to see that the parametrizations of
   $\gamma_1$ and $\gamma_2$ by arc-length satisfy 
   the assumptions of Theorem \ref{thm:cauchy}, say with
   $\delta_1=\delta_2=0$, and 
   $T$ being the domain of definition of $g_j$. Hence
   $f_j(z)$ is indeed defined and extends to 
   an entire function with the stated asymptotics. 
   The fact that $f_j$ belongs to
   the class $\B$ follows from Proposition \ref{prop:logtract}. 
\end{proof} 

\section{The hyperbolic metric of simply-connected domains}
\label{sec:hyperbolicmetric}
  We frequently use the \emph{hyperbolic metric} in a domain
   $U\subset\C$ that omits more than two points. (For an introduction to the
   hyperbolic metric, see e.g.\ \cite{beardonminda}.) We denote distance
   with respect to this metric by $\dist_{U}$, and the density of the
   metric by $\rho_U$. That is,
    \[ \dist_{U}(z,w) = 
  \inf_{\gamma} \int_{0}^{1} |\gamma'(t)|\rho_U(\gamma(t))\,dt, \]
    where the infimum is taken over all curves $\gamma:[0,1]\to U$ with
   $\gamma(0)=z$ and $\gamma(1)=w$. 

  We shall routinely use a number of standard facts about the hyperbolic metric.

 \begin{prop}[(Properties of the hyperbolic metric)] \label{prop:hyp_metric} 
   \mbox{}
   \begin{enumerate}[(a)]
    \item The hyperbolic metric in the right half plane $\HH$ is given by
      $\rho_{\HH} = \frac{1}{\re z}$. In particular,
      $\dist_{\HH}(1,x)=\log x$ for $x\geq 1$ and
      $\dist_{\HH}(x,x+ix)\leq 1$ for every $x>0$. 
      \label{item:halfplane}
    \item In the strip $S=\{|\im z|<\pi\}$, we have   
         $\dist_S(z,w)\geq |\re z -\re w |/2$ for all $z,w\in S$. 
      \label{item:strip}
    \item If $V\subset W$, then $\rho_{V}(z)\geq \rho_W(z)$ for all $z\in V$.
     \label{item:pick}
    \item If $V,W\subset\C$ are hyperbolic
      and $f:V\to W$ is a conformal isomorphism, then $f$ is a hyperbolic 
      isometry; i.e. $\rho_V(z)=|f'(z)|\cdot \rho_W(f(z))$. 
       \label{item:isomorphism}
    \item If $V\subset\C$ is simply connected, then
            $1/(2\dist(z,\partial V))\leq \rho_V(z)\leq 2/\dist(z,\partial V)$ for all $z\in V$. \label{item:simplyconnected}
   \end{enumerate}
  \end{prop}

 Let us make two more simple observations about the hyperbolic metric
  in simply-connected domains.
  \begin{lem}[(Hyperbolic distance and Euclidean distance)]\label{lem:distance}
   Let $V\subset\C$ be a simply-connected domain, and let $z,w\in V$.
    Then
    \[ \dist_V(z,w)\geq \frac{1}{2}\log\left(1+\frac{|z-w|}{\dist(z,\partial V)}\right). \]
  \end{lem}
 \begin{proof}
  Set $\delta := \dist(z,\partial V)$. 
  Let $\gamma:[0,T]\to V$ be a curve connecting $z$ and $w$,
   parametrized by Euclidean arc-length. So $T\geq \dist(z,w)$. Then we have
   $\dist(\gamma(t),\partial V)\leq \delta+t$. Thus, by
    Proposition \ref{prop:hyp_metric} (\ref{item:simplyconnected}), 
   \begin{align*} \notag
     \int_0^T |\gamma'(t)| \rho_V(t)\, dt &\geq
       \int_0^T \frac{dt}{2\dist(\gamma(t),\partial V)} \geq
       \frac{1}{2}\int_0^T \frac{dt}{\delta+t}  \\
       &= \frac{1}{2}\left(\log(\delta+T)-\log(\delta)\right) =
         \frac{1}{2}\log\left(1+\frac{T}{\delta}\right).\qedhere
\end{align*} 
 \end{proof}

 \begin{lem}[(Bounded hyperbolic diameter of Euclidean disks)]
    \label{lem:disks}
  Let $V\subset\C$ be a simply-connected domain, let $z_0\in V$ and
   let $\Delta\in(0,2]$. Define
   $\delta := \Delta\cdot \dist(z_0,\partial V)/4$.

  If $z\in V$ with $|z-z_0|\leq \delta$, then
    $\dist_V(z,z_0)\leq \Delta$.
 \end{lem}
 \begin{proof}
  Set $d:=|z-z_0|$ and let
    $\gamma:[0,d]\to V$ be the straight line segment connecting
    $z_0$ and $z$, parametrized by arc-length. 
    Then 
    \[ \dist(\gamma(t),\partial V)\geq \frac{\dist(z_0,\partial V)}{2}, \]
    and thus, again using
    Proposition \ref{prop:hyp_metric} (\ref{item:simplyconnected}), 
    \[
       \dist_V(z,z_0)\leq \int_{0}^d \rho_V(\gamma(t))dt \leq
         \int_0^d\frac{4dt}{\dist(z_0,\partial V)} =
            \frac{4\cdot d}{\dist(z_0,\partial V)} \leq \Delta.\qedhere
     \]
 \end{proof}

 Finally, we will on occasion use the following version of the 
  Ahlfors distortion theorem \cite[Corollary to Theorem 4.8]{ahlforsconformal}.
 
 \begin{thm}[(Ahlfors distortion theorem)] \label{thm:ahlforsdistortion}
  Let $V\subset\C$ be a simply connected domain, and let
   $z,w\in V$ with $a := \re z < \re z =: b$. 
   Let $\sigma_z, \sigma_w\subset V$ be the
   maximal vertical line segments passing through $z$ resp.\ $w$.

  Set $S=\{a+ib: |b|\leq\pi\}$, and let 
   $\phi:V\to S$ be a conformal isomorphism such that
   $\phi(\sigma_z)$ and $\phi(\sigma_w)$ both separate $-\infty$ from
   $+\infty$ in $S$ (i.e., they connect the upper and lower boundaries
   of the strip $S$), and such that $\phi(\sigma_z)$ separates
   $\phi(\sigma_w)$ from $-\infty$ (i.e., $\phi(\sigma_z)$ is to
   the left of $\phi(\sigma_w)$ in $S$). 
   
  For $aleq x \leq b$, let $\theta(x)$ denote the 
   shortest length of
   a vertical line segment at real part $w$ that
   separates $z$ from $w$ in $V$. 
 
  If $\int_a^b dx/\theta(x) \geq 1/2$, then
    \[ \phi(b) - \phi(a) \geq 2\pi \int_a^b \frac{dx}{\theta(x)} - 2\log 32. \]
 \end{thm}

 We also note the following fact, which is closely related to
  the distortion theorem:

 \begin{lem}[(Geodesics in quadrilaterals)] \label{lem:modulus}
  Let $V$ be a simply-connected domain that is symmetric with respect to the
  real axis. Let $\gamma_1$ and $\gamma_2$ be two cross-cuts of $V$ that
  are symmetric with respect to the real axis, with 
  $\gamma_1\cap \gamma_2=\emptyset$, 
  and suppose that the quadrilateral $Q$ bounded by $\gamma_1$ and
  $\gamma_2$ in $V$ has modulus at least $1/2$. (I.e., the extremal
  length of the family of curves connecting $\gamma_1$ and $\gamma_2$ in
   $V$ is at least $1/2$.)

  Then $\cl{Q}$ contains a geodesic of $V$ that is symmetric with respect to 
   the real axis.
 \end{lem}
 \begin{proof}
   Let $S$ denote the strip $\{a+ib: |b|\leq \pi\}$ and let
    $\phi:V\to S$ be a conformal isomorphism that takes 
    $\R\cap V$ to the real axis. Set $\wt{Q} := \phi(Q)$; then
    $\wt{Q}$ is a quadrilateral in $S$, symmetric with respect to the
    real axis, of modulus at least $1/2$. We must show that
    $\wt{Q}$ contains a vertical segment connecting the two boundary 
    components of $S$. 

   The exponential map takes $\wt{Q}$ to an annulus of modulus 
   at least $1/2$, slit along an interval of the positive real axis,
    which separates $0$ from $\infty$. By Teichm\"uller's modulus
    theorem \cite[Theorem 4-7]{ahlforsconformal}, the closure of this annulus contains
    a round circle centered at the origin, which completes the proof. 
 \end{proof}

\section{Approximation of model functions} \label{sec:approx}

 We now turn to proving Theorem \ref{thm:approx}. As already mentioned,
  this result is ``best'' possible with our method,
  in the sense that the domain $H$ is chosen as
  close to a right half plane as possible while still guaranteeing
  convergence of the Cauchy integral. However, sometimes it is convenient
  to use other image domains, e.g.\ because it might be possible to
  write down an explicit mapping function for these. 
  We will therefore work in a somewhat
  more general setting. In particular, we recover the results of
  \cite{strahlen} as a special case. 

 \begin{standingassumption}[(Assumption on $H$ and $\gamma$)]%
    \label{ass:initialconfiguration}
  $H\subset\C$ is a simply connected domain containing the right half plane
   $\HH$. Furthermore, 
     \[ \alpha:(-\infty,\infty)\to H \] 
   is a piecewise
   smooth injective curve for which there 
   exist positive constants $A_1$, $A_2$, $A_3$, $A_4$ and $\Delta$ with $A_1,A_4>1$ such that, 
    for all $t\in\R$:
   \begin{enumerate}[(a)]
    \item $\re \alpha(t)\leq -13 \log_+(t)+\log A_1$.\label{item:left}
    \item $|\alpha'(t)|\leq A_2$. (If $t$ belongs to the discrete set
       where $\alpha$ is not differentiable, this means that both the left and
       right derivatives are bounded by $A_2$.) 
       \label{item:derivative}
    \item $\dist_H(\alpha(t),1+it)\leq A_3$. (Recall that
       $\dist_H$ denotes the
       hyperbolic distance in $H$.)
      \label{item:hypdist}
    \item If $\zeta\in H$ with $\dist_H(\alpha(t),\zeta)\leq \Delta$, then 
                   $\re\zeta \leq -4\log_+|t| + \log A_4$.
      \label{item:smalldisk}
   \end{enumerate}
  We refer to the pair of $H$ and $\alpha$ as the \emph{initial configuration}.
   The bounds below will depend on this initial choice.
 \end{standingassumption}
 \begin{remark}[Remark 1]
   The final two conditions may seem somewhat technical. Roughly, they
    mean that the curve $\alpha$ stays within a comparable distance from both
    $\partial H$ and the line $\{\re\zeta=1\}$; compare 
   Section~\ref{sec:initial}. 
%%% QUESTION: If first three conditions are satisfied,
%%%    can we always obtain the final one by restricting H?
 \end{remark}
 \begin{remark}[Remark 2]
  It is not difficult to see that the choice
    \[ H := \{x+iy: x>-14 \log_+|y|\} \]
  used in the statement of Theorem \ref{thm:approx} and the
  curve
    \[ \alpha(t) := it - 13\log_+|t| + 1 \]
  satisfy our standing assumption. For completeness, we 
  provide the argument in Section \ref{sec:initial}.
 \end{remark}
 \begin{remark}[Remark 3]
  In applications, the domain $H$ and the curve $\alpha$ will be
   fixed, so dependence on the initial configuration will not
   usually be important. However, we note 
   that our bounds will depend only on the constants
   $A_1$ to $A_4$ and $\Delta$, but not otherwise on $\alpha$ and $H$.
 \end{remark}

 \begin{standingassumption}[(Model function)]
   \label{ass:tract}
 Furthermore, 
    \[ \Psi: T\to H \]
  is a model function in the sense of Definition
  \ref{defn:modelfunction} (where $H$ is the domain from 
  Standing Assumption \ref{ass:initialconfiguration}). We additionally assume, by way of normalization, that
  that $1\in T$, $0\in \partial T$, $\dist(1,\partial T)=0$ and $\Psi(1)=1$.  

  Let $V$ be a component of $\exp^{-1}(T)$ and let
   $G:V\to H$ be the conformal isomorphism 
   $G:=\Psi\circ\exp$. We also set $g := \exp\circ\Psi$. Note that
   we have $g\circ\exp=\exp\circ G$.

 Finally, we set $\beta := G^{-1}\circ\alpha$ and
   $\gamma := \exp\circ\beta=\Psi^{-1}\circ\alpha$. 
   Let $\tilde{T}$ be the component of
   $\C\setminus\gamma$ that is contained in $T$.
 \end{standingassumption}

We will now show that (under these assumptions), we can apply Theorem \ref{thm:cauchy}
 to $T$ and a reparametrization of $\gamma$.

\begin{lem}[(Growth and distance to boundary)] \label{lem:hypdist}
  There are constants $M_1$ and $M_2$, depending only on
   the initial configuration, such that 
     \[ |\gamma(t)|\leq M_1 \cdot |t|_+^4 \quad\text{and}\quad 
      \dist(\gamma(t),\partial T)\geq 
           M_2\cdot |t|_+^{-4} \]
   for all $t\in\R$. 
\end{lem} 
\begin{proof}
We set $C:=A_3+1$. Using the fact that hyperbolic distances in $H$ are smaller than those 
  in the half plane $\HH$ (recall Proposition \ref{prop:hyp_metric}), we see that
   \begin{align*}
         \dist_{H}(1,\alpha(t)) &\leq
          \dist_{\HH}(1,|t|_+)+
          \dist_{\HH}(|t|_+,|t|_++ti) \\
        &\quad + \dist_{\HH}(|t|_++ti,1+ti) 
        +\dist_H(1+ti,\alpha(t)) \\ &\leq
          \log_+ |t| + 1 + \log_+|t| + A_3 =
           C + 2\log_+|t|.
   \end{align*}
   
  We now use the Ahlfors distortion theorem,  Theorem
   \ref{thm:ahlforsdistortion}, 
   to deduce the desired estimate. Let $\sigma_1$ be the maximal vertical
   line segment in $V$ containing $0=G^{-1}(1)$, and let $\sigma_2$ be the maximal
   vertical line segment containing $\beta(t)$. Set 
   $S=\{a+ib:|b|<\pi\}$ and let $\phi:V\to S$ be a conformal 
   isomorphism such that $\re \phi(0)=0$, such that
   $\phi(\sigma_1)$ and $\phi(\sigma_2)$ both connect the upper and lower
   boundaries of $S$, and such that $\phi(\sigma_2)$ is to the right
   of $\phi(\sigma_1)$. This is always possible: pick two prime ends
   $\zeta_1$ and $\zeta_2$ (if $V$ is a Jordan domain, this simply
   means picking two points on $\partial V$) such that $\sigma_1$ separates
   $\zeta_1$ from $\sigma_2$ and $\sigma_2$ separates $\sigma_1$ from
   $\zeta_2$. We then choose $\phi$ such that $\phi(\sigma_1)=-\infty$
   and $\phi(\sigma_2)=+\infty$. 
 
  Recall that
   $V$ does not intersect its own translates by integer multiples of $2\pi i$,
   and hence does not contain any vertical segments of height 
   greater than $2\pi$.
    If
    $\re\beta(t)\geq \pi$, then Theorem
    \ref{thm:ahlforsdistortion} 
    and Proposition \ref{prop:hyp_metric} (\ref{item:strip}) imply that 
   \[ 2\dist_S(\phi(0),\phi(\beta(t))) \geq 
        |\phi(\beta(t)) - \phi(0)| \geq
          \re\beta(t) - D, \]
    where $D=2\log 32$ is a universal constant. Thus
     \[ \re\beta(t) \leq \max(\pi,D + 2C) + 4\log_+|t|. \]
     Recalling that $\gamma(t)=\exp(\beta(t))$, the first claim is proved.

  Similarly, we can estimate 
   $\dist(\gamma(t),\partial T)$, using Lemma \ref{lem:distance}. 
   Indeed, set $z := \gamma(t)$ and 
    $\delta := \dist(z,\partial T)$. Recall that $\dist(1,\partial T)=1$, so 
    $|z-1| \geq 1 - \delta$, and hence
    \[ C+2\log_+|t|\geq \dist_H(1,\alpha(t)) = 
        \dist_T(1,z)\geq \frac{1}{2}\log\left(1+\frac{|z-1|}{\delta}\right) \geq 
          \frac{1}{2}\log\frac{1}{\delta} \]
    by Lemma \ref{lem:distance}. Exponentiating this inequality and
     rearranging, we see that
      \[ \delta \geq (|t|_+)^{-4}\cdot \exp(-2C), \]
   as desired. 
\end{proof}

\begin{cor} \label{cor:deriv}
  There is a constant $M_3$, depending only on the initial configuration, such that
    \[ |\gamma'(t)| \leq M_3\cdot |t|_+^4 \]
   for all $t\in\R$.

  Furthermore, there is a constant $M_4$, depending only on the
   initial configuration, with
   the following property. If $t\in\R$ and $z\in\C$ with 
   \[ |z-\gamma(t)|\leq M_4\cdot |t|_+^{-4}, \]
   then $z\in T$ and $|g(z)|\leq A_4\cdot |t|_+^{-4}$. 
\end{cor}
\begin{proof}
We have 
     \[ |\gamma'(t)| = |\alpha'(t)|\cdot |(G^{-1})'(\alpha(t))|\cdot 
                          \exp(\re\beta(t)). \] 
  The first term is bounded by our standing assumption that
    $|\alpha'(t)|\leq A_2$. To estimate the second term, we use hyperbolic
    geometry: $G$ is a conformal isomorphism and $V$ and $H$ are both
    simply connected, so
     \[ (G^{-1})'(\alpha(t)) =
           \frac{\rho_{H}(\alpha(t))}{\rho_V(\beta(t))} 
   \leq 4\frac{\dist(\beta(t),\partial V)}{\dist(\alpha(t),\partial H)}. \]
     Since $\exp$ is injective on $V$,
     we have  $\dist(z,\partial V)\leq \pi$ for all $z\in V$.
    We also note that $\dist(\alpha(t),\partial H)\geq 1/D$ for
     some constant $D$ that depends only on $A_3$. Indeed, if
     $\re \alpha(t)\geq 1/2$, there is nothing to prove (since
     $\HH\subset H$). Otherwise, we have
     $\dist(\alpha(t),1+it)\geq 1/2$ and 
     $\dist_H(\alpha(t),1+it)\leq A_3$, and the claim follows from
     Lemma \ref{lem:distance}. 

  Finally, 
     \[ \exp(\re\beta(t)) = 
        |\gamma(t)|\leq M_1 \cdot |t|_+^4 \]
  by 
    Lemma \ref{lem:hypdist}.
  Combining these estimates, we see that
    \[ |\gamma'(t)| \leq 
      4\pi\cdot A_2\cdot D \cdot M_1\cdot  |t|_+^4. \]

  To prove the second claim, let us assume without loss of generality that
   $\Delta\leq 2$ and set
\[      M_4 :=  \frac{M_2\cdot \Delta}{4}. \]
   Suppose that $t\in\R$ and $z\in\C$ are as in the claim; then 
\[
       |z-\gamma(t)|\leq M_4 \cdot |t|_+^{-4}  
          = \frac{\Delta}{4}\cdot M_2 \cdot |t|_+^{-4}
\leq \frac{\Delta}{4}\cdot \dist(\gamma(t),\partial T) \]
    by Lemma \ref{lem:hypdist}. Hence we can apply Lemma \ref{lem:disks} to see that
    $\dist_T(z,\gamma(t)) \leq \Delta$, and thus
    $\dist_H(\Psi(z),\alpha(t))\leq\Delta$.  By the standing assumption, it follows that
   \[
     |g(z)| = \exp(\re \Psi(z)) \leq A_4\cdot |t|_+^{-4}. \qedhere
    \]
\end{proof}

 We now ready to apply Theorem \ref{thm:cauchy} to conclude:

\begin{cor}[(Approximation by entire functions)] \label{cor:approximation}
  In the setting of Standing Assumption \ref{ass:initialconfiguration},
  let $\Psi:T\to H$ be any model function. Define $g(z) := \exp(\Psi(z))$, $\gamma := \Psi^{-1}\circ\alpha$ and $z_0 := \Psi^{-1}(1)$. Then 
    \[ h(z) := \frac{1}{2\pi i}\int_{\gamma} \frac{g(\zeta)}{\zeta-z}d\zeta  \]
   defines a holomorphic function for $z\notin\gamma$. This function satisfies
     \[ |h(z)|\leq M_5 \quad\text{and}\quad |h(z)| \leq \max(|z_0|_+,\dist(z_0,\partial T))\cdot \frac{M_6}{|z|_+} \]
   for all $z$. Here the constants $M_5$ and $M_6$ depend only on the initial configuration. 

   Furthermore, 
    \[ f(z) := \begin{cases}
                h(z)+g(z) & z\in\tilde{T} \\
                 h(z) & z\notin \tilde{T} \end{cases} \]
    extends to an entire function $f:\C\to\C$ with
    $S(f)\subset\cl{\D_{2M_5}(0)}$. 
\end{cor}
\begin{proof}
 Let us first assume that $T$ and $\Psi$ are normalized as in Standing Assumption \ref{ass:tract}; i.e.,\
   $0\in \partial T$, $\dist(1,\partial T)=1$ and $\Psi(1)=1$. 

 We reparametrize the curve $\gamma$ by the substitution
   \[ \tau := \begin{cases}
                t^4 & t\geq 1 \\
                t & |t|<1 \\
                -t^4 & t\leq -1. \end{cases} \] 
 By Lemma \ref{lem:hypdist} and Corollary \ref{cor:deriv}, we have 
   \begin{enumerate}[(a)] 
     \item $|\gamma(\tau)| \leq 
                M_1\cdot |\tau|_+$;
     \item $|d\gamma(\tau)/d\tau|= \frac{1}{4}\cdot |\gamma'(t)|\cdot |\tau|_+^{-3/4} \leq
                    \frac{M_4}{4}\cdot |\tau|_+^{1/4}$;
     \item $|g(\gamma(\tau))|=\exp(\re\alpha(t))\leq 
             A_1\cdot |t|_+^{-13} = A_1\cdot |\tau|_+^{-(2+1/4+1)}$;
     \item 
        if $|z-\gamma(\tau)|\leq M_4\cdot |\tau|_+^{-1}$, then
        $z\in T$ and $|g(z)|\leq A_4\cdot |\tau|_+^{-1}$. 
   \end{enumerate} 
 
 Thus the claims on the convergence and asymptotics of $h(z)$ and the analytic continuation 
  of $f(z)$ follow from Theorem \ref{thm:cauchy}. The claim regarding the singular values
  of $f$ follows from Proposition \ref{prop:logtract}.

 For general $T$ and $\Psi$, we normalize and apply the case that was just established. More precisely, 
  set $z_0 := \Psi^{-1}(1)$ and let $a\in \partial T$ be a point whose distance to $z_0$ is minimal;
  define $\alpha := z_0-a$. Then the function 
   \[ \tilde{\Psi}(z) := \Psi\left(\alpha z + a\right) \]
  satisfies Standing Assymption \ref{ass:tract}. Let $\tilde{h}$ be the corresponding function
     \[ \tilde{h}(z) =  \frac{1}{2\pi i}\int_{\tilde{\gamma}} \frac{\tilde{g}(\zeta)}{\zeta-z}d\zeta =
            h(\alpha z + a), \]
   where $\tilde{\gamma}(t)=(\gamma(t)-a)/\alpha$. As we have just seen, 
      \[ |h(z)| = \left|\tilde{h}\left(\frac{z - a}{\alpha} \right)\right|\leq \frac{M_5}{|(z-a)/\alpha|_+} \leq M_5 \]
    for a constant $M_5$ depending only on the initial configuration. It is elementary to verify that
       \[ \left|\frac{z-a}{\alpha}\right|_+ \geq \frac{|z|_+}{2\max(|a|_+,|\alpha|_+)}. \]
  (E.g., distinguish between the cases $|z|<2|a|_+$ and $|z|\geq 2|a|_+$.) 
    Hence 
     \[ |h(z)| \leq 2\max(|a|_+,|\alpha|_+)\cdot \frac{M_5}{|z|_+} \leq 4\max(|z_0|_+,|\alpha|)\cdot \frac{M_5}{|z|_+}, \]
   as desired. 
\end{proof}
%
% For the record:
%  If |z+a|<= 1, then |z|_+ <= 1+|a|, and hence the claim follows.
%  If |z+a|>=1, and |z|<1+|a|, then |z+a| >=1 >= 1/(1+|a|).
% If |z+a|>= 1 and |z|>=1+|a|, then |z+a| >= |z| - |a| = |z|/(1+|a|) + |a||z|/(1+|a|) - |a| >= |z|/(1+|a|).
%
% If |\alpha| >=1, then |\alpha z |_+ \geq |z|_+.
% If |\alpha| < 1 and |z|>= 1, then |\alpha z|_+ >= |\alpha z| = |\alpha| |z|_+.
% If |\alpha| < 1 and |z| < 1, then |\alpha z|_+ = 1 and |\alpha| |z|_+ = |\alpha|<1.

\begin{rem}[(Dependence of the bounds in Theorem {\ref{thm:approx}})]
 \label{rem:bounddependence}
 In the next section, we shall carry out the simple verification that the choice of $H$ and $\gamma$ from Theorem
   \ref{thm:approx} satisfies Standing Assumption \ref{ass:initialconfiguration}. Hence Corollary \ref{cor:approximation} implies
    the theorem.

 In particular, we note that the approximating function $f$ from Theorem \ref{thm:approx} satisfies
   $|f(z)|\leq M$ outside $T$ and $|f(z)-g(z)|\leq M$ in $T$, where $M=M_5$ is a universal constant. 
   This fact will be used in the second part of the paper. 
\end{rem}

\section{Valid initial configurations} \label{sec:initial}

 \begin{prop}
  Let $\rho:\R\to [0,\infty)$ be a continuous function that is
   increasing on $[0,\infty)$ and decreasing on $(-\infty,0]$. Set
    \[ H := \{z\in\C: \re z > -\rho(\im z)\}. \]
   Suppose that 
     \[ \alpha:(-\infty,\infty)\to H \] 
   is a piecewise
   smooth injective curve such that 
   (\ref{item:left}) and (\ref{item:derivative}) of Standing
   Assumption \ref{ass:initialconfiguration} are satisfied.
   If furthermore
    $\im \alpha(t) = t$ for all $t\in\R$ and  
     \[ c\cdot |\re \alpha(t) | \leq
          \dist(\alpha(t),\partial H) \leq C |\re\alpha(t)| \]
    when $|t|$ is sufficiently large, where $c$ and $C$ are positive constants,
   then Standing
   Assumption \ref{ass:initialconfiguration} is satisfied. 
 \end{prop}
 \begin{proof}
  Observe first that, if $a_1,a_2,b\in\R$ are such that
   $-\rho(b)<a_1<a_2$, then
    \[ \dist(a_1+bi,\partial H) < \dist(a_2+bi,\partial H). \]

  Now let $t\in\R$; by
   (\ref{item:left}) of Standing Assumption \ref{ass:initialconfiguration},
   we may assume that $|t|$ is sufficiently large that
   $\re\alpha(t)<-1$. Then
    \begin{align*} \dist_H(\alpha(t),1+it) &\leq
        \int_{0}^{1-\re\alpha(t)} \rho_H(\alpha(t)+x)dx \leq
       \int_{0}^{1-\re\alpha(t)} \frac{2\,dx}{\dist(\alpha(t),\partial H)} \\ &\leq
       2\frac{1-\re\alpha(t)}{c\cdot |\re\alpha(t)|} \leq
       \frac{4}{c}.
    \end{align*}
 So requirement (\ref{item:hypdist}) holds when $|t|$ is large enough. 
  
  Finally, let $\Delta>0$. Then by Lemma
   \ref{lem:distance}, the hyperbolic disk of radius $\Delta$ around 
   $\alpha(t)$ is contained in the Euclidean disk around $\alpha(t)$ of radius
   \[ (e^{2\Delta}-1)\cdot \dist(\alpha(t),\partial H). \]
  On the other hand, we have  
   \[ -\re\alpha(t) > 13 \log_+|t| - \log A_1 > 8\log_+|t|, \]
  provided $|t|$ is sufficiently large, and hence
    \[ -\re\alpha(t) - 4\log_+|t| > -\re\alpha(t)/2 > \frac{\dist(\alpha(t),\partial H)}{2C}. \] 
  So if we choose $\Delta := \log(1+1/(4C))/2$, we have 
    \[ \re \zeta < -4\log_+|t| \]
  whenever $\dist_H(\zeta,\gamma(t))\leq \Delta$ (still under the assumption
  that $t$ is sufficiently large). Hence requirement (\ref{item:smalldisk}) also holds when $|t|$ is sufficiently large.

 We obtain the requirements for all $t\in\R$ by choosing  $A_3$ and
   $A_4$ sufficiently large.
 \end{proof}

 \begin{cor}[(Two valid initial configurations)] \label{cor:initial}
  Set either
    \begin{align*} H &:= \{x+iy: x>-14 \log_+|y|\} \quad\text{and}
   \\ \alpha(t) &:= it - 13\log_+|t| + 1 
    \end{align*}
  or
   \begin{align*} 
      H &:= \{ x + iy: x > -2M\cdot |y| \} \quad\text{and} 
   \\ \alpha(t) &:= it - M|t| + 1 \end{align*}
   (where $M>0$ is a constant). 

 Then $H$ and $\alpha$ satisfy the conditions of Standing Assumption
   \ref{ass:initialconfiguration}.
 \end{cor}
 \begin{proof}
  By the previous proposition, we only need to check 
   (\ref{item:left}) and (\ref{item:derivative}). The first of these is
   self-evident. The second is also immediate; in the first case we have
    \[ \alpha'(t) = \begin{cases}
                       i - \frac{13}{t} & |t| > 1 \\
                       i   & -1<t<1. \end{cases}   \]
   and in the second  
    \[
      \alpha'(t) = i-M\cdot \sign(t).\qedhere
     \]
 \end{proof}

\section{Quasiconformal equivalence and conjugacy}
\label{sec:conjugacy}
  We begin with a simple lemma on obtaining a quasiconformal map on a 
   vertical strip that
   extrapolates between the identity on one boundary and a map that is
   not too far from the identity on the other.

  \begin{lem} \label{lem:qcstrip}
   Let $R>0$, and let $\tau:\R\to\C$ be a differentiable function with the property
    that there is $\eta<1/2$ with
    $|\tau'(t)|\leq \eta$ and $|\tau(t)| \leq \eta R$ for all $t\in\R$. 

   Let $S$ denote the strip $\{a+ib: 0<a<R\}$. Then there
    is a quasiconformal map $\theta$, defined on $S$, such that
    $\theta(it) = it$ and
      \[ \theta(R + it) = R + it+ \tau(t) \]
    for all $t\in\R$.

   Furthermore, the complex dilatation $\mu_{\theta}=(d\theta/d\bar{z})/(d\theta/dz)$ is bounded by
    $\eta/(1-\eta)$ almost everywhere. 
  \end{lem}
 \begin{proof} 
   The map $\theta$ is defined simply by linear interpolation along
    horizontal line segments:
    \begin{equation}\label{eqn:theta}
     \theta(z) := z + \frac{\re z }{R}\cdot \tau(\im(z)). \end{equation} 
   This map satisfies the required boundary conditions. Clearly $\theta$ is injective
   when restricted to a fixed horizontal line segment $\{z\in S: \im z = t\}$, and this
   line segment is mapped to the straight segment $L_t$ that connects $it$ and
   $R + it + \tau(t)$. In order to prove that $\theta$
    is a homeomorphism onto its image, we must show that no two of these image
    line segments intersect. 

  This follows from the assumptions on $\tau$. Indeed, let $t_0\in\R$. 
    From $|\tau(t_0)|/R\leq \eta<1/2$, 
    we see that the line segment $L_{t_0}$ is sloped at an
    angle strictly between $-\pi/4$ and $\pi/4$. On the other hand, from 
     $|\tau'|\leq \eta < 1/2$
    we see that the argument of the derivative
      \[ \frac{d\theta(R+it)}{dt} = i + \tau'(t) \]
    lies strictly between $\pi/4$ and $3\pi/4$. This implies that, for $t>t_0$, the 
     point $\theta(R+it)$ lies above the line through $it$ and $R+it$, and for $t<t_0$
     it lies below this line. Hence the line segments $L_t$ and $L_{t_0}$ do not
     intersect for $t\neq t_0$.
 
   So it remains to estimate the complex dilatation of $\theta$. 
   If we set $h(z) := \frac{\re z}{R}\cdot \tau(\im(z))$, then we have
     \[ \frac{\partial \theta}{\partial z} = 1 + \frac{\partial h}{\partial z}
     \quad\text{and}\quad\frac{\partial \theta}{\partial \bar{z}} = \frac{\partial h}{\partial \bar{z}}. \]
    Furthermore, writing $z=x+iy$, we have $h(z)=\tau(y)\cdot x/R$, and hence 
     \[ \frac{\partial h(z)}{\partial x} = \frac{\tau(y)}{R}\quad
\text{and} \quad \frac{\partial h(z)}{\partial y} = \frac{x}{R}\cdot \tau'(y). \]
    Thus
   \[ \left|\frac{\partial h(z)}{dz}\right|,
\left|\frac{\partial h(z)}{d\bar{z}}\right| \leq
    \frac{1}{2}\left(\frac{|\tau(y)|}{R} + |\tau'(y)|\right) \leq \eta. \]
   
   Hence we have seen that 
      \[ \left|\frac{\partial\theta}{\partial z}\right|\geq
              1 - \eta \quad\text{and}\quad 
        \left|\frac{\partial \theta}{\partial \bar{z}}\right|\leq \eta. \]
    It follows that $\theta$ is quasiconformal and satisfies the
     stated bound on its dilatation. 
 \end{proof}

 As in \cite{boettcher}, it will be useful to work in logarithmic
  coordinates when proving our equivalence and conjugacy statements. 
   Hence we shall initially state our results for
  the following class of functions introduced in
  \cite{boettcher,strahlen}.

 \begin{defn}[(The class $\BlogP$)]\label{defn:Blog} 
   A holomorphic function 
    \[ F:\V\to H \]
   is said to belong to the class $\BlogP$ if 
   \begin{enumerate}[(A)] 
    \item $H$ is a $2\pi i$-periodic unbounded Jordan domain that contains
     a right half-plane.   \label{item:H}
    \item $\V\neq\emptyset$ 
     is $2\pi i$-periodic and $\re z$ is bounded from below in
     $\V$.
    \item $F$ is $2\pi i$-periodic. \label{item:periodicity}
    \item Each component $T$ of $\V$ is an 
      unbounded Jordan domain that is disjoint
      from all its $2\pi i\Z$-translates. 
      For each such $T$, the restriction
      $F:T\to H$ is a conformal 
      isomorphism with $F(\infty)=\infty$. 
      ($T$ is called a \emph{tract of $F$}; we 
       denote the inverse of $F|_T$ by 
       $F_T^{-1}$.) \label{item:tracts}
    \item The components of $\V$ accumulate only at $\infty$; i.e.,
      if $z_n\in\V$ is a sequence of points no two of which belong to the same
      component of $\V$, then $z_n\to\infty$.  \label{item:accumulatingatinfty}
   \end{enumerate}
 \end{defn}
 \begin{remark}[Remark 1]
  In \cite{boettcher}, the class of functions described in Definition \ref{defn:Blog}
   is simply called
   $\mathcal{B}_{\log}$, while in \cite{strahlen}, that notation is
   used for the larger set obtained by omitting the periodicity
   requirement (\ref{item:periodicity}). Subsequent papers such 
   as \cite{devaneyhairs} followed
   the latter convention, hence we use $\BlogP$ for the class above.
 \end{remark}
 \begin{remark}
  Let $f\in\B$ and let $R>0$ be sufficiently large to ensure that
   $\D_R(0)$ contains the set $S(f)\cup \{0\}\cup \{f(0)\}$. Set
   $W := \C\setminus\overline{\D_R(0)}$, 
   $H := \exp^{-1}(W) = \{z\in\C:\re z > \log R\}$ and 
   $\V := \exp^{-1}(f^{-1}(W))$. Then every component $V$ of $\V$ is
   a Jordan domain whose boundary passes through infinity, and
   $f\circ\exp:V\to W$ is a universal covering. Hence we can define a function
   $F:\V\to H$ that belongs to $\BlogP$ and satisfies 
   $\exp\circ F = f\circ \exp$. Such a function is called
   a \emph{logarithmic transform of $f$} (or ``$f$ in logarithmic coordinates''); this is the motivation for the definition of $\BlogP$. 
 \end{remark}

 In the introduction, we stated our results only for 
  functions with a single tract. However, the equivalence and conjugacy results
  in this section actually hold for functions in the class $\BlogP$ that
  are sufficiently close to each other, even if there are infinitely many tracts. 
  The key statement is about
  \emph{quasiconformal equivalence}:
 
 \begin{thm}[(QC equivalence in the class $\BlogP$)] \label{thm:Blogequivalence}
  Suppose that $G:\V_G\to \HH$ and 
   $F:\V_F\to H$ belong to the class $\BlogP$, and that there is
   a constant $M>0$ with 
    \[ \V_G\supset \V_F \supset \{ G^{-1}(z) : \re z > M \} \]
   and $|F(z)-G(z)|\leq M$ for all $z\in \V_F$.

  Then for every $R\geq 4(M+2\pi)$, there exists a quasiconformal map
   $\Phi:\C\to\C$, commuting with translation by $2\pi i$, such that
    \[ F(\Phi(z)) = G(z)\quad\text{whenever $\re G(z)\geq R$}. \]

  Moreover, 
    the complex dilatation $\mu_{\Psi}$ of $\Psi$ satisfies
      \[ |\mu_{\Psi}(z)|\leq 4\cdot \frac{M}{R} \]
   for almost all $z\in\C$.

  Furthermore, $\Phi(z)=z$ when $z\notin \V_G$ and when $z\in \V_G$ with
    $\re G(z)\leq R/2$, and 
    \[ |\Phi(z) - z| \leq 
          \sup\{|F(\zeta)-G(\zeta)|: |\zeta - z|\leq K \}\leq M \]
   otherwise, where $K=2\pi\cdot (1+\log 2)$ is a universal constant. 
 \end{thm}
 \begin{proof} 
   We begin with a simple observation regarding the structure of $H$ and $\V_F$.
 \begin{claim}
    The range $H$ of $F$ contains the half plane 
    $\{z\in\HH: \re z > 2M\}$. Furthermore,
     if $V$ is a component of $\V_G$, there is a unique component $\wt{V}$
    of $\V_F$ contained in $V$.
 \end{claim}
 \begin{subproof}
  Let $V$ be a component of $\V_G$. Then, by assumption, there is a component
   $\wt{V}$ of $\V_F$ that contains the connected set
   $G_V^{-1}(\{\re z > M\})$, and this component is contained in $V$. 
   Now $\re G$ is bounded on $V\setminus\wt{V}$, and hence $\re F$ is bounded
   on $(V\cap \V_F)\setminus \wt{V}$. Since $\re F$ is unbounded on every connected
   component of $\V_F$, we have seen that indeed $\wt{V}=V\cap \V_F$.

  Furthermore, let $z\in\HH$ with $\re z > 2M$ and
  consider the circle
   $C$ of radius $M$ around $z$. Then $G|_V^{-1}(C)$ is a simple closed curve
   in $\wt{V}$, and it follows from the assumption that its image under
   $F$ winds once around $z$. Since $H$ is simply-connected,
   it follows that $z\in H$ as claimed.
 \end{subproof}

\begin{claim}
   Let $V$ and $\wt{V}$ be as above. Then there is a quasiconformal map
    $\phi_V:\HH\to\HH$ with $\phi_V(z)=z$ when
    $\re z \leq R/2$ 
 and
    $\phi_V(z)=G(F|_{\wt{V}}^{-1}(z))$ when
    $\re z \geq R$. Furthermore, the complex dilatation of $\phi_V$ is bounded by
    $4M/R<1$. 
\end{claim}
\begin{subproof}
  We construct this map using Lemma \ref{lem:qcstrip}. Indeed,
   set $h(z) := G(F_{\wt{V}}^{-1}(z))-z$. 
   Then $|h(z)|\leq M$ for all $z\in \V_F$. Furthermore, if $z\in\HH$ with
   $\re z\geq R$, then the domain of $h$ contains 
   the disk of radius $R-2M\geq R/2$. Hence by the Cauchy inequaliy,
   we see that 
     \[ |h'(z)|\leq \frac{2M}{R}. \]

  Thus, if we set $\tau(t) := h(it+R)$, we have
   \begin{align*}
      |\tau'(t)|\leq \frac{2M}{R} < \frac{1}{2}\quad\text{and}\quad
      \frac{2|\tau(t)|}{R}\leq \frac{2M}{R} < \frac{1}{2} 
   \end{align*}
  for all $t$. Hence we can apply Lemma \ref{lem:qcstrip} to obtain
  a quasiconformal map $\theta$ on
  the strip between real parts $R/2$ and $R$, such that $\theta$ is the identity
   on the left boundary of the strip and agrees with
    $G(F_{\wt{V}}^{-1}(z))=z+\tau(z)$ on the right boundary. Hence 
    \[ \phi_V(z) := \begin{cases}
                      z & \text{if } \re z \leq R/2\\
                      \theta(z) & \text{if } R/2 < \re z < R \\
                      G(F_{\wt{V}}^{-1}(z)) & \text{if } \re z \geq R
                    \end{cases} \]
  is the desired quasiconformal homeomorphism. (Note that the map is
  quasiconformal near points with real parts equal to $R$ or $R/2$, since
  a straight line is quasiconformally removable.) The bound on the dilatation also
  follows from Lemma \ref{lem:qcstrip}. 
 \end{subproof}

 Now let $z\in \C$. If $z\notin \V_G$, then we define
  $\Phi(z) := z$. Otherwise, let $V$ be the component of $\V_G$ containing
   $z$ and define
   \[ \Phi(z) := G_V^{-1}(\phi_V(G(z))). \]
  Then $\Phi:\C\to\C$ is a homeomorphism. Furthermore, $\Phi$ is
  quasiconformal on $\V_G$ and agrees with the identity outside
  this set. Hence, the map is quasiconformal everywhere by Royden's
  glueing lemma (\cite[Lemma 2]{bersglueing}, \cite[Lemma 2]{polylikemaps}), 
  and satisfies the stated dilatation bound. If
  $\re G(z)\geq R$ and $V$ is the component of $\V_G$ containing $z$, then
    \[ F(\Phi(z)) = F(G_V^{-1}(\phi_V(G(z)))) = G(z) \]
   by construction.

  To prove the final statement, we first observe that 
     \[ |G'(z)|\geq 1 \]
   whenever $\re G(z)\geq 4\pi$ by
       \cite[Lemma 1]{alexmisha}. If $z\in \wt{V}$ with
    $\Phi(z)\neq z$, then by construction we have
    $\re G(z)\geq R/2\geq 4\pi$ and, likewise, $\re G(\Phi(z))\geq 4\pi$. 
    Thus
    \[ |z - \Phi(z)|\leq |G(z) - G(\Phi(z))| =
                   |G(z) - \phi_V(G(z))|, \]
    where $V$ is the component of $\V_G$ containing $z$. 

  If $\re G(z)\geq R$, let us set $\omega := G(z)$; otherwise we set
   $\omega := R + i\cdot \im(G(z))$. Then we have
     \[ |G(z) - \phi_V(G(z))|  \leq |h(\omega)| \]
   by the definition of $\phi_V$, provided we  use the formula
    (\ref{eqn:theta}) for $\theta$ from Lemma \ref{lem:qcstrip}.
  
  Writing $z_1 := F_{\wt{V}}^{-1}(\omega)$, we have seen that
    \[ |z-\Phi(z)|\leq |h(\omega)| = 
          |G(z_1) - F(z_1)|. \]
  It remains to estimate $|z-z_1|$. The hyperbolic distance
   in $\HH$ between $G(z)$ and $\omega$ is at most $\log 2$. 
   Furthermore, the Euclidean distance in $\HH$ between 
   $\omega=F(z_1)$ and $G(z_1)$ is bounded by $M$. Since
   $\re \omega \geq R\geq 2M$, the hyperbolic distance
   in $\HH$ between these two points is at most $1$. 

  So the hyperbolic distance between $G(z)$ and $G(z_1)$ is bounded by
   $1+\log 2$, and thus the hyperbolic distance in $V$ between
   $z$ and $z_1$ is also bounded by this constant. Using the standard estimate
   on the hyperbolic metric, and the fact that $V$ does not intersect its
   translates by multiples of $2\pi i$, we see that indeed
  \[
     |z-z_1| \leq 2\pi \cdot (1+\log 2).\qedhere
   \]
 \end{proof}

\begin{proof}[Proof of Theorem \ref{thm:equivalence}]
  We can let $F,G\in\BlogP$ be logarithmic transforms of $f$ and $g$,
   respectively. More precisely, we assume without loss of generality that
   $0\notin T$ and set
   $G := \Psi\circ \exp$, which (suitably restricted) is an element
   of $\BlogP$ with domain $\V_G$ and range $\HH$.\footnote{%
   This need not quite be true if the boundary of the domain of $\Psi$ is not
   a Jordan curve, since the boundary of the range of $\Psi$ has the interval
   from $-i$ to $i$ in common with $\HH$. This problem is easily dealt with by
   first conjugating $g$ by $z\mapsto e^{\eps}\cdot z$ for some small $\eps>0$.}
    If we choose $\mu>0$
   sufficiently large, then we can likewise define a map
   $F:\V_F \to \{z\in\HH:\re z > \mu\}$ that belongs to the class $\BlogP$ 
   and satisfies $\exp\circ F = f\circ\exp$ and $\V_F\subset \V_G$. Then
   $F$ and $G$ satisfy the hypotheses of the previous theorem.

 In fact, recall that
   $|f(z)-g(z)|\leq C/|z|$ for a suitable constant $C$ (provided $z$ is
   sufficiently large). 
 Since the exponential map is expanding on a right half plane, it follows
  that we can choose the logarithmic transform $F$ in such a way that
    \begin{equation} \label{eqn:estimateinlogcoordinates}
    |F(z)-G(z)| \leq C\cdot e^{-\re z} \end{equation}
   for all $z\in \V_F$.

 Now choose $R$ sufficiently large that we can apply Theorem \ref{thm:Blogequivalence}, and
   let $\Phi$ be the quasiconformal map obtained from the theorem.
   Then
   $\phi(e^z) := e^{\Phi(z)}$ defines a quasiconformal map that satisfies 
   $g(z)=f(\phi(z))$ whenever $|g(z)|\geq e^R$.
  To estimate the asymptotics of $\phi$ at $\infty$, observe that 
    \[ |\Phi(z) - z | \leq C\cdot e^{K}\cdot e^{-\re z}, \]
   where $K$ is the universal constant from the previous theorem. Let us
   set $w := e^z$. If $\re z$ is sufficiently large, we have 
    \begin{align*}|w- \phi(w)| &=
       |w|\cdot |1 - \frac{e^{\Phi(z)}}{e^z}| \leq  
       |w|\cdot (1 + e^{\re(\Phi(z) - z)}) \\ &\leq
       2|w|\cdot |\Phi(z)-z|\leq 2|w|\cdot C\cdot e^{K}\cdot e^{-\re z} = 
          2Ce^{K}. \qedhere
    \end{align*}
\end{proof}

\begin{thm}[(QC conjugacy in the class $\BlogP$)] \label{thm:Blogconjugacy}
  Suppose that $F$, $G$, $M$ and $R$ are as in Theorem
   \ref{thm:Blogequivalence}. Suppose furthermore that 
    \[ \V_G\subset  \{z\in\HH: \re z > R \}. \]

  Then there exists a quasiconformal homeomorphism $\Theta:\C\to\C$,
   commuting with translation by $2\pi i$, such that 
    \[ \Theta(G(z)) = F(\Theta(z)) \]
   whenever $z\in\V_G$ with $\re G(z) \geq R$. 

   The complex dilatation $\mu_{\Theta}$ 
    satisfies $|\mu_{\Theta}|\leq 4M/R$ almost everywhere, and
    $\mu_{\Theta}=0$ almost everywhere on the Julia set $J(G)=\{z\in \V_G: G^n(z)\in \V_G\text{ for all $n$}\}$.

  Furthermore, suppose that
   $|F(z)-G(z)|\to 0$ uniformly as
   $\re z\to\infty$. Then
    \[ \sup_{z\in J_Q(G)} |\Theta(z)-z| \to 0 \]
   as $Q\to\infty$, where
     \[ J_Q(G) = \{z\in J(G): \re G^j(z)\geq Q\text{ for all $j\geq 0$} \}. \]
\end{thm} 
\begin{proof}
 This theorem essentially follows from the corresponding results in
  \cite{boettcher}. However, for completeness we shall sketch the proof,
   which is not difficult in our case.

  Let $\Phi$ be the map from Theorem \ref{thm:Blogequivalence}.
  We define a sequence of quasiconformal
    maps $\Phi_j$ by $\Phi_0 := \Phi$ and 
    \[ 
       \Phi_{j+1}(z) := \begin{cases}
      F_{\wt{V}}^{-1}(\Phi_j(G(z))) & \text{if }z\in\V_G\text{ and } \re G(z) > R \\
         \Phi(z)  & \text{otherwise}. \end{cases}\]
  If $z\in \V_G$ and $\re G(z)=R$, then $G(z)\notin \V_G$, and hence 
    \[ F_{\wt{V}}^{-1}(\Phi_j(G(z))) = F_{\wt{V}}^{-1}(G(z)) = 
          \Phi(z). \]
   Hence the maps match up on the boundary, and each $\Phi_j$ is a homeomorphism.
   It follows from Royden's 
   glueing lemma that the $\Phi_j$ are all quasiconformal, with the
   same bound on the dilatation as $\Phi$. Furthermore, because
   all $\Phi_j$ agree on $\C\setminus \V_G$, it follows by induction
   that $\Phi_j$ and $\Phi_{j+1}$ agree on
    $\C\setminus G^{-j}(\V_G)$. Hence the sequence of maps 
   stabilize on the open set $\C\setminus J(G)$. This set is dense in $\C$ because
   $G$ is expanding with respect to the hyperbolic metric of $\HH$ (compare \cite[Lemma 2.3]{boettcher}). Together with
   the compactness property of quasiconformal maps, this implies that
   $\Phi_j\to\Theta$ for a quasiconformal map $\Theta$. By construction, 
   $\Theta\circ G = F\circ\Theta$ whenever $z\in\V_G$ and $\re z \geq R$.  

  Each map $\Phi_{j}$ is conformal on a neighborhood of $J(G)$, implying the
   statement about the dilatation on $J(G)$.

  The last claim follows easily from the fact that $G$ is expanding on
   $J_Q(G)$ and the final statement in Theorem \ref{thm:Blogequivalence}.
\end{proof}

\begin{proof}[Proof of Theorem \ref{thm:conjugacy}]
  Assuming $\rho_0$ was chosen sufficiently large, 
   we apply Theorem \ref{thm:approx} to obtain a function $f$ approximating $g$.
   Recall that $|f(z)-g(z)|\leq \mu$ on $\T$, where $\mu>1$ is a universal constant
    (Remark \ref{rem:bounddependence}). 

 As in the proof of Theorem \ref{thm:equivalence} and of Proposition \ref{prop:logtract}, we can let 
    $G:\V_G\to \HH$ and $F:\V_F\to \{a+ib : a>\log(2\mu)\}$ be logarithmic transforms of $f$ and $g$, respectively. Furthermore,
    $F$ can be chosen such that 
    \[ |F(z) - G(z)| \leq M \]
   for a universal constant $M$, and we can assume that $M$ is chosen so large that
     $\V_F\supset \{G^{-1}(z) : \re z > M\}$. In other words, the hypotheses of Theorem \ref{thm:Blogequivalence} are satisfied. 

  If $\rho_0$ was chosen sufficiently large, then $R := \log \rho_0$ satisfies $R\geq 4(M+2\pi)$, and hence we can apply Theorem
    \ref{thm:Blogconjugacy} to obtain a conjugacy $\Theta$ between $G$ and $F$. 
   Defining $\theta(\exp(z)) := \exp(\Theta(z))$ yields the desired
   conjugacy between $g$ and $f$.
\end{proof}

\begin{rem}[(Additional properties of the conjugacy)]
 It follows from the proof that, in the setting of Theorem
  \ref{thm:conjugacy}, the following additional statements hold: 

 \begin{itemize}
  \item  The quasiconformal
    dilatation of $\theta$ tends to zero
   as $\inf_{z\in T}|z| \to \infty$.
  \item 
 $\displaystyle{\sup_{z\in J_Q(g)} d_{\log}(z,\theta(z)) \to 0}$
   as $Q\to\infty$, where
    \[ J_Q(g):= \{z\in J(g): |g^n(z)|\geq Q\text{ for all $n\geq 1$} \} \]
    and 
    $d_{\log}$ denotes the distance with respect to
   the metric $\frac{|dz|}{|z|}$.
 \end{itemize}
\end{rem}

\section{Functions with full hyperbolic dimension}
\label{sec:hypdim}
 We now turn to proving Theorem \ref{thm:hypdim}. That is,
  we construct an entire function $f\in\B$ that is
  hyperbolic and whose Julia set contains hyperbolic sets
  of Hausdorff dimension arbitrarily close to two.  
  To do so, we will construct a suitable model function 
  $\Psi:T\to H$, where $H$ is as in Theorem \ref{thm:approx}, and apply  
  Theorem \ref{thm:conjugacy}. 

 \subsection*{Idea of the construction}
 Before we give the details, let us broadly outline the idea. 
  For simplicity, let us consider models $\Psi:T\to \HH$ (where $\HH$ is the right
  half plane). It turns out that changing to the domain $H$ from Theorem
  \ref{thm:approx} does not add significant new issues. 

 We discuss the construction in logarithmic coordinates. 
   Suppose that $\Psi:T\to\HH$ is a model function, with $\overline{T}\subset\HH$,
   and let $V$ be a component
   of $\V := \exp^{-1}(T)$; we define
   $G := \Psi\circ\exp$. Then $G:\V\to\HH$ is $2\pi i$-periodic and $G|_V$ is a conformal
   isomorphism between $V$ and $\HH$. 

 The basic set-up of the construction is somewhat reminiscent of the
   proof \cite{baranskikarpinskazdunik} that the hyperbolic dimension of a function
   with a logarithmic tract over infinity is always strictly greater than one. Let
   $K>0$ be sufficiently large, and let $Q$ be the square
   $Q=\{a+ib: K < a < 3K; |b|<K\}$, centered at the point
   $2K$. We shall build a finite iterated function system (compare \cite{urbanskimauldin} or also \cite[Definition 2.10]{hypdim}) on
   the square $Q$, each of whose branches is of the form
   $z\mapsto G^{-1}(z+2\pi i k)$, for some $k\in\Z$ and some branch of
   $G^{-1}$. Then the union of all $2\pi i\Z$-translates of the limit set of
   this function system is invariant under $G$. Projecting by the exponential map,
   we hence obtain a hyperbolic set for the map $g := \exp\circ \Psi$;
   the question is how to construct the tract $T$ and the function system in 
   such a way that the Hausdorff dimension of this limit set is close to two.

  Suppose that we are given points $\omega_1,\dots,\omega_m\in V$ with
    $G(\omega_j) = 2K + 2\pi i k_j$, for some $k_j\in\Z$, such that
    the $\omega_j$ have real parts between $K$ and $3K$. Suppose furthermore that
    $|k_j-k_i|>K/\pi$ for $i\neq j$. 
   Then, for each $j$, there are approximately
     $K/\pi$ points $\omega_j^{\ell}$ in $\omega_j + 2\pi i \Z$ that are
     themselves contained in the square $Q$. For each such point, we can define
     a conformal map $\phi_j^{\ell}$ on $Q$ by $\phi_j^{\ell}(z) = G^{-1}(z+2\pi i k_j)$,
     where the branch of $G^{-1}$ is chosen such that $\phi_j^{\ell}(2K) = \omega_j^{\ell}$. 

   These maps form a conformal iterated function system on $Q$ 
      (assuming that each $\phi_j^{\ell}(Q)$ does not intersect the boundary of $Q$,
       which will not be difficult to ensure). Note that each $\phi_j^{\ell}$ is
      a contraction, and the contraction factor is on the order of
          \[ \frac{\rho_{\HH}(2K)}{\rho_V(\omega_j)} 
                  \approx \frac{\dist(\omega_j,\partial V)}{K}.\] 
    Hence the size of the 
     contraction factor depends on the distance of $\omega_j$ to the
     boundary of $V$. So we should try to construct the tract $V$ in 
     such a way that the curve 
  \[ \Gamma_P := \{z\in V: \re G(z) = P\}, \]
     where $P=2K$, 
     isn't always too close to the boundary of $V$. 

\begin{figure}
 \begin{center}
    \resizebox{.98\textwidth}{!}{\input{figure_approx_hypdim_2.tex}}
  \end{center}
 \caption{The domain $V = V\bigl((\eps_k)_{k\in\N}\bigr)$\label{fig:hypdimtract}}
\end{figure}

 We will show that it is possible to ensure that 
   $\Gamma_P$ stays a fixed distance away from $\partial V$ at
   regular intervals. 
  The construction  depends
  on a sequence $\Xi=(\eps_k)_{k\in\N}$ of numbers $\eps_k\in [0,1]$. The 
  domain $V$
  consists of a central strip of fixed height $2h<2\pi$, joined to a sequence of equally
  spaced, equally sized chambers (on both sides) of width $w$. 
  The connection between these
  chambers is opened by a fraction of
  $\eps_k$. That is, if $\eps_k=0$, then the chamber is completely closed
  off, whereas if $\eps=1$, the chamber is completely open; see Figure
  \ref{fig:hypdimtract}.

  When a chamber is completely open, i.e.\ $\eps_k=1$, then---provided $k$ is
    large enough---the curve $\Gamma_P$ will run close to the boundary of the 
    chamber. On the other hand, if $\eps_k=0$, then the chamber is completely
    closed off, and hence $\Gamma_P$ cannot enter it. By continuity, it is
    possible to ensure that $\Gamma_P$ runs through the central point of the
    chamber. This suggest that, for suitable choice of the sequence $\Xi$, all the
    chambers between real parts $K$ and $3K$ will contain a point
    $\omega_j$ with $\dist(\omega_j,\partial V)>\delta$, for some fixed $\delta$,
    and with $G(\omega_j) = 2K + 2\pi i k_j$, as above. The number of points
    $\omega_j^{\ell}$ then is roughly $2K^2/w\cdot \pi$; i.e. grows
    quadratically with $K$, while the contraction factor is 
    of order $1/K$. Thus the Hausdorff dimension of the corresponding limit set
    tends to $2$ as $K\to\infty$. 
   By making sure that the above properties hold for a sequence $K_j$ tending to
    infinity, the hyperbolic dimension of the resulting function is equal to two.

\subsection*{The hyperbolic metric in $H$}
 Before we provide the  details of the construction that was just outlined, let 
   us make some observations about the hyperbolic metric of 
   \[     H := \{x+iy: x>-14 \log_+|y|\}. \]
 \begin{lem}[(Hyperbolic geometry of $H$)] \label{lem:metricH}
  The segment $[0,\infty)$ is a hyperbolic geodesic in $H$. Furthermore, 
   there is a constant $C_1>1$ such that, 
   for every $z_0\in H$, 
   the hyperbolic geodesic of $H$ that contains $z_0$ and is perpendicular to
   (and symmetric with respect to) the real axis is contained in 
    $\{z\in H: |z_0|/C_1 < |z|<C_1 |z_0|\}$. 

  Furthermore, there is a constant $C_2$ with the following property. 
   If $x\geq 2$ and $z\in H$ with $1\leq \re z \leq x/2$, then
    \[ \dist_H(x,z) \geq \frac{\log\left(\frac{x}{\re z}\right)}{C_2}. \]
 \end{lem}
\begin{proof}
  The first claim is clear because the domain is symmetric with 
  respect to the real axis. 

 To prove the second claim, 
   let us use the term ``vertical geodesic'' to refer to geodesics that
   are perpendicular to the real axis. 
  Consider the quadrilateral in $H$ bounded by the arcs
  $\sigma_{|z_0|/C_1}$ and $\sigma_{|z_0|}$, where
  $\sigma_t = \{z\in H: |z|=t\}$ and $C_1>1$. By the comparison principle
   for extremal length, the modulus of this quadrilateral
  is greater than that of the slit annulus 
     \[ \{z\in\C: 1/C_1 < |z|/|z_0| < 1, z\notin (-\infty,0) \}. \]
   The latter modulus is equal to 
  $(\log C_1)/2\pi$. Hence, by Lemma \ref{lem:modulus},
  if $C_1 > e^{\pi}$, then this quadrilateral contains a vertical geodesic 
  of $H$. For the same reason, there is a vertical geodesic between
  $\sigma_{|z_0|}$ and $\sigma_{C_1 |z_0|}$. The vertical geodesic passing through
  $z$ must lie between these two geodesics, and hence lies between
  $\sigma_{|z_0|/C_1}$ and $\sigma_{C_1 |z_0|}$, as claimed. 

  To verify the final claim, let 
   $z = P + iy$, and assume without loss of generality that $y\geq 0$. 
   Applying Lemma \ref{lem:distance}, we see that
    \[ \dist_H(x,z) \geq \frac{1}{2}\log \left(1+\frac{|x-z|}{\dist(z,\partial H)}\right). \]
  We have $\dist(z,\partial H) \leq P + 14 \log_+y$.    If $\log_+y \leq P$, then
     \[ \frac{|x-z|}{\dist(z,\partial H)} \geq 
          \frac{x-P}{15P} \geq \frac{x}{30P} \geq \frac{1}{30}\sqrt{\frac{x}{P}}. \]
    On the other hand, if $\log_+y > P$, then 
     \begin{align*} 
      \frac{|x-z|}{\dist(z,\partial H)}  &\geq
         \frac{\max(x-P,y)}{15\log y} \geq
         \frac{\max(x-P,y)}{15\log(\max(x-P,y))} \\ &\geq 
           \frac{\sqrt{\max(x-P,y)}}{15} \geq 
           \frac{\sqrt{x-P}}{15} \geq \frac{\sqrt{x}}{30} \geq 
           \frac{1}{30} \sqrt{\frac{x}{P}}. \end{align*}
    (Here we used that $t/\log t \geq \sqrt{t}$ for $t>1$.)   So in either case we have
     \[ \dist_H(x,z) \geq \frac{1}{2} \log\left(1+\frac{1}{30}\sqrt{\frac{x}{P}}\right), \]
    which implies that we can set
  \[
    C_2 := \sup_{t\geq 2} \frac{2\log t}{\log(1+\sqrt{t}/30)} < \infty.\qedhere
   \]
\end{proof}

\subsection*{Description of $V$ and parameter selection} 
  In a slight modification of the construction described above, we will
   allow the parameter sequence $\Xi=(\eps_k)_{k\in\N}$ to take
   values $\eps_k\in [0,1]\cup \{1^*\}$. Here $\eps_k=1^*$ will mean that
  the chamber is not only completely open,
  but if furthermore also $\eps_{k+1}=1^*$, then the wall between the two chambers
  is removed. The reason for this is that we wish to show that our example
  can be chosen in such a way that the Julia set has positive measure,
  and this requires us to introduce long parts of the tract that have height $2\pi$.
  (Readers interested only in an example with hyperbolic
  dimension equal to two can ignore this possibility in the following.)

 We fix the 
  width $w:=2\pi$ of the chambers and also set 
  $h := \pi/3$. Then,
    using the convention that $|1^*|=1$, 
  $V$ is given by
   \begin{align*}
     V := V\bigl(\Xi) := &\left\{a+ib: a> \frac{1}{2}, |b|<h\right\} \\
               &\cup \bigcup_{k\in\N:\eps_k\neq 0}\left\{ a+ib: |a-kw|<\frac{w}{2},
                                            h<|b|<\pi \right\}\\
               &\cup \bigcup_{k\in\N}\left\{a+ib: |a-kw|<\frac{|\eps_k|\cdot w}{2},
                                       |b|=h\right\} \\
               &\cup \bigcup_{k\in\N: \eps_k=\eps_{k+1}=1^*} \left\{a+ib: a=\frac{(2k+1)w}{2}, |b|<\pi\right\}.
   \end{align*}
 Let $G:V\to H$ be a conformal isomorphism with
    $G(1)=1$ and $G'(1)>0$. Because the tract is symmetric with respect
    to the real axis, we have $G([1,\infty))=[1,\infty)$, and
    $G(\bar{z})=\overline{G(z)}$. The dynamical model function $\Psi$ we later approximate will be
    given by $\Psi(\rho_0\cdot \exp(z)) := G(z)$, where $\rho_0$ is the constant from 
    Theorem \ref{thm:conjugacy} and $\Xi$ is a suitably chosen sequence. 

 We now proceed to investigate the behavior of the function $G$, for a given
   sequence $\Xi=(\eps_k)_{k\in\N}$. 
 All definitions in the following depend on $\Xi$, but
 for simplicity of notation, we often suppress this dependence. 
 We also emphasize that 
  any constants appearing in the results will be
  \emph{independent} of $\Xi$, unless explicitly stated otherwise. 

 A key fact is that the tract $V(\Xi)$ 
  depends continuously on $\Xi$, using the
  product topology on $([0,1] \cup \{1^*\})^{\N}$ and the Carath\'eodory 
  kernel topology
  for the domains. In particular, the inverse $G^{-1}:H\to V$, and thus also the 
  function $G$ itself, depends
  continuously
  on $\Xi$ in the topology of locally uniform convergence. 

 We begin by estimating $|G(z)|$ independently of $\Xi$:
 \begin{lem} \label{lem:sizeofG}
  There are constants $C_3>1$ and $C_4>0$ such that
    \[ \frac{\re z}{C_3} - C_4 \leq \log |G(z)| \leq C_3\re z + C_4\]
  for all $z\in V$ with $\re z \geq w$.
 \end{lem}
\begin{proof}
 By the standard estimate on the hyperbolic metric (or, alternatively, by the Ahlfors
   distortion theorem), it follows that there is a constant $C_3>1$
  such that
   \[ x/C_3 \leq \log G(x) \leq C_3 x \]
  for all $x\in V\cap \R$ with $x\geq w/2>1$. 

 Now let $z\in V$ with $\re z \geq w$.  We consider the vertical geodesic $\gamma$ 
  of $V$ passing through $z$ (i.e., the unique geodesic through z that intersects the 
 real axis  perpendicularly.)
  Let $x$ be the point of intersection of $\gamma$ with the real axis. It follows
  from Lemma \ref{lem:metricH} that
  $ G(x) /C_1 \leq |G(z)|\leq C_1 \cdot G(x)$; hence
  \begin{equation}\label{eqn:Gzestimate}
    x/C_3 - \log C_1\leq \log|G(z)|\leq C_3x + \log C_1. \end{equation}
  
 Furthermore,  let $k\geq 2$ be maximal with
  $kw/2 \leq \re z$. Then 
  the quadrilateral in $V$ bounded by the vertical cross-cuts 
   $\{\zeta\in V: \re \zeta = (k-1)w/2\}$ and $\{\zeta\in V: \re \zeta = kw/2\}$ 
   has modulus at least $1/2$ (by choice of $w$ and the comparison principle for
   extremal length). Hence, by Lemma
  \ref{lem:modulus} this quadrilateral contains a vertical geodesic of $V$, and
   in particular we must have $x \geq (k-1)w/2$. Similarly we have 
     $x\leq (k+2)w/2$, and hence
  \[ |x-\re z| \leq w. \]
  Combining this with (\ref{eqn:Gzestimate}), the proof is complete.
\end{proof}

 For $P>0$, define as above 
  \[ \Gamma_P := \{z\in V: \re G(z) = P\}. \]
 We shall define \emph{signed hyperbolic distance} from $\Gamma_P$ in 
  $V$ by setting 
   \[ \delta(z,P) := \begin{cases}
                        -\dist_V(z,\Gamma_P) &\text{if } \re G(z) > P; \\
                  0  & \text{if } \re G(z) = P \\
                  \dist_V(z,\Gamma_P) &\text{if } \re G(z) < P; \\
               \infty & \text{if }z\notin V. \end{cases} \]
 We observe that $\delta(z,P)$ depends continuously on $\Xi$ for fixed $z$ and $P$. Hence,
  setting $\zeta_k := kw + i\cdot 2\pi/3$, we see that
   \[ \delta_{P,k}(\Xi) := \delta(\zeta_k,P) \]
  is a continuous function of $\Xi$.

 \begin{lem} \label{lem:k(P)}
  \begin{enumerate}
   \item For all $k\in\N$ and all
     $\Xi = (\eps_k)_{k\in\N}$:  if $\eps_k=0$, then $\delta_{P,k}(\Xi) = \infty$.
   \item For all $P>0$, 
      there is $k_0=k_0(P)\in\N$ (independent of the sequence $\Xi=(\eps_k)_{k\in\N}$)
     such that, for all $k\geq k_0$: if
           $\eps_k=1$ (or $\eps_k=1^*$), then $\delta_{P,k} \leq -1$.

     Furthermore, there is a constant $\kappa_2$ such that
     $k_0(P)\leq \kappa_2\cdot \log P$ for sufficiently large $P$. 
   \end{enumerate}
 \end{lem}
 \begin{proof}
  The first part is trivial by definition.

 To prove the second part, we can assume without loss of generality that $P\geq 1$.
  Suppose that 
   $\eps_k=1$ or
   $\eps_k=1^*$ with $k\geq C_3\cdot \log(2P)/w$, where $C_3$ is as in
   the proof of Lemma \ref{lem:sizeofG}. Set $x := \re \zeta_k = k\cdot w$.

 By the standard estimate on the
   hyperbolic metric, the hyperbolic distance between $\zeta_k$ and
   $x := \re \zeta_k=k\cdot w$ is bounded by some uniform constant $C$. (In fact, given our
   choice of $w$ and $h$, we can take $C=4$.) Hence, if $\delta_{P,k}(\Xi)>-1$, then
   the hyperbolic distance between $x$ and $\Gamma_P$ is bounded by
    $C+1$. This is only possible if $k$ is sufficiently small.

 Indeed, we have $G(x) \geq \exp(x/C_3) \geq 2P$ by choice of $k$. Applying
  Lemma \ref{lem:metricH}, we see that 
    \[ C+1 \geq \dist_V(x,\Gamma_P) \geq 
        \frac{\log\left (\frac{G(x)}{P}\right)}{C_2} \geq
        \frac{kw}{C_2\cdot C_3} - \frac{\log P}{C_2}.\]
   The claim follows by rearranging.
 \end{proof}
\begin{remark}
 We could have replaced the sequence $(\xi_k)$ by any sequence
  whose hyperbolic distance from $\xi_k$ is bounded (or does not grow
  too quickly). This would yield a stronger version of Theorem \ref{thm:selection} below,
  but we will not require this extra generality. 
\end{remark}

 We can now prove our parameter selection result. 

 \begin{thm} \label{thm:selection}
   Let $A\subset\N$, and suppose we are given values $(\wt{\eps}_k)_{k\in\N\setminus A}$ with $\wt{\eps}_k\in [0,1]\cup\{1^*\}$, and a sequence 
   $(P_k)_{k\in A}$ with $k\geq k_0(P_k)$ for all $k$. 

   Then there is a sequence $\Xi = (\eps_k)_{k\in\N}$ such that
    \begin{itemize}
     \item $\eps_k = \wt{\eps}_k$ for $k\notin A$ and
     \item $\delta_{k,P_k}(\Xi)=0$ (i.e., 
       $\zeta_k\in \Gamma_{P_k}$) for all $k\in A$. 
    \end{itemize}
 \end{thm}
 \begin{proof}
  Let us first prove the result for \emph{finite} subsets
   $A\subset\N$. If $m:=\#A=1$, then
   the claim simply corresponds to the intermediate value theorem.
   
   For $m>1$, the claim similarly follows by basic topology. 
    More precisely, let us define a map from the $m$-cube 
    $[0,1]^A$ to itself. For $x\in [0,1]^A$, let $\Xi(x)$ be the sequence
    defined by
    setting $\eps_k=x_k$ for $k\in A$ and $\eps_k=\wt{\eps}_k$ for
    $k\notin A$.

  For $k\in A$ and $x\in [0,1]^A$, we define
    \[ \wt{\delta_k}(x) := \begin{cases}
                      0 & \delta_{k,P_k}(\Xi(x)) > 1, \\
                      1  & \delta_{k,P_k}(\Xi(x)) < -1, \\
                   \frac{1-\delta_{k,P_k}(\Xi(x))}{2} & \text{otherwise}.\end{cases} \]
   Then $\phi: (\eps_k)_{k\in A}\mapsto (\wt{\delta_k})_{k\in A}$ is
   a continuous map of the $m$-cube to itself, and by Lemma
   \ref{lem:k(P)} it maps any face (of any dimension)
   to itself. Hence the map between the homotopy (or homology) groups 
   of the boundary of the cube induced by $\phi$ is the identity. This implies that
   $\phi$ must be surjective; and thus there exists $x\in [0,1]^A$ such that
    $\wt{\delta_k}(x)=1/2$ for all $k\in A$.

 This proves the theorem for finite $A$. If $A$ is infinite, we take 
  an increasing sequence of finite subsets $A_k$ that exhaust $A$, and
  let $\Xi$ be a limit of the corresponding sequences. The claim follows by continuity
  of the functions $\Xi\mapsto \delta_{k,P_k}(\Xi)$.
 \end{proof} 

\subsection*{Proof of Theorem {\ref{thm:hypdim}}}
 We are now in a position to complete the construction, in line with the sketch we gave
   at the beginning of the section. Recall that we will need to move the tract sufficiently
   far to the
   right in logarihmic coordinates in order to apply Theorem \ref{thm:conjugacy};
   i.e., we are really interested in the dynamical behavior of the function
   $G(z-R_0)$, where $R_0:=\log \rho_0>0$.
 \begin{cor}
  There is a constant $K_0>R_0$ 
    with the following property. Suppose that
    $K_0\leq K_1<K_2<\dots$ is a sequence of integers with
    $K_{i+1}>3K_i$ for all $i\geq 1$. Then there is a sequence
    $\Xi$ so that the corresponding function
    $G:V(\Xi)\to H$ satisfies 
  \[  \re G(\zeta_k) = 2K_i \]
    for all $k$ with
    $K_i < w\cdot k + R_0 < 3K_i$ and such that $\eps_k=1^*$ for all other $k$. 
 \end{cor}
 \begin{proof}
   This follows from the previous theorem. It only needs to be checked
    that $(K_i- R_0)/w > k_0(2K_i)$, provided $K_0$ was chosen
    sufficiently large, and this follows from the statement on the 
    size of $k_0$ in Lemma \ref{lem:k(P)}. 
 \end{proof}

 For the
  remainder of the section, let us fix
  $1\leq K_0 < K_1 < K_2 < \dots$ and let
  $V=V(\Xi)$ be as in the preceding corollary. We define
     \[ G_0(z) := G(z - R_0), \quad 
   \Psi(\exp(z)) := G_{0}(z)\quad\text{and}\quad g(z) := \exp(\Psi(z)). \]
   So $G_0:V'\to H$ is a conformal isomorphism, where $V' = V + R_0$, and
   $\Psi:T\to H$ is a model function, where $T := \exp(V')$.

 By definition, $T\subset \{|z|>\rho_0\}$, and hence we can apply Theorem
   $\ref{thm:conjugacy}$ to $\Psi$. We obtain a disjoint-type entire function $f:\C\to\C$ with
   $|g(z)-f(z)|= O(1/|z|)$ and such that $f$ and $g$ are quasiconformally conjugate on their Julia sets. 

   \begin{thm} \label{thm:hypdiminproof}
     The hyperbolic dimension of $g$ is two. Hence also $\hypdim(f)=2$. 
      Both $f$ and $g$ have finite order.
   \end{thm}
 \begin{proof}
   Let us fix $i$ (sufficiently large) and $K:=K_i$; we shall construct a hyperbolic set whose dimension tends to two as $i$ tends to $\infty$.

   Let $k^-$ and $k^+$ be the minimal resp.\ maximal values of $k$ with
    $K+w/2 \leq w\cdot k + R_0 \leq 3K-w/2$. That is,
  \[ k^- := \left\lceil \frac{2K+w-2R_0}{2w}\right\rceil \quad\text{and}\quad
     k^+ := \left\lfloor\frac{6K-w-2R_0}{2w}\right\rfloor \] 
Set $\zeta_k' := \zeta_k + R_0$.
    Then $\re G_0(\zeta_k') = 2K$ by choice of $V$; observe also that
      $\im G_0(\zeta_k')>0$ (since $G_0$ is real on the real axis, and 
     the points $\zeta_k'$ lie above the real axis in $V$). Let us define
       \[
          m_k := \left \lfloor \frac{\im G_0(\zeta_k')}{2\pi}\right\rfloor\quad\text{and}\quad
        \omega_k := G_0^{-1}( 2K + 2\pi i m_k). \]
     In other words, $G_0(\omega_k)$ is the $2\pi i\Z$-translate between $2K$ and $G_0(\zeta_k')$ that is
     closest to $G_0(\zeta_k')$. 

    Let $Q$ be the square of sidelength $K$ centered at $2K$. 
     For $k^- \leq k \leq k^+$, define
       \[ \phi_k : Q\to V; z\mapsto G_0^{-1}(z + 2\pi i m_k). \]
    We begin by showing that the $\phi_k$ do not contract too strongly.
    \begin{claim}[Claim 1]
      There exists a universal constant $\lambda_0$ such that
        $|\phi_k'(z)|\geq \lambda_0/K$ for all $k\in\{k^-,\dots,k^+\}$ and all $z\in Q$. 
    \end{claim}
    \begin{subproof}
     Each $\phi_k$ extends conformally to a square of sidelength $2K$ centered at $2K$. 
     Hence it
      suffices to estimate the derivative of $\phi_k$ at the center of the square; the
      claim then follows from Koebe's distortion theorem. We have
        \begin{equation}  \label{eqn:contraction}
     |\phi_k'(2K)| = \frac{\rho_H(G_0(\omega_k))}{\rho_{V'}(\omega_k)} \geq 
                \frac{\dist(\omega_k,\partial V')}{4\dist(G_0(\omega_k),\partial H)}. 
      \end{equation}

     We first note that the hyperbolic distance between $\omega_k$ and $\zeta_k'$ is uniformly bounded 
      (and in fact tends to zero as $i$ tends to infinity). By choice of $\Xi$, 
     $\dist(\zeta_k', V')$ is uniformly
      bounded from below; hence $\dist(\omega_k, V')>\delta_1$ for some universal $\delta_1>0$.
     Furthermore,
      \begin{align*}
         \dist(G_0(\omega_k),\partial H ) &\leq 2K + 14\log_+(2\pi m_k) \leq
                       2K + 14\log_+|G(\zeta_k)| \\
          &\leq
              2K + C_3\cdot \re \zeta_k + C_4 \leq
              2K + 3 C_3 K + C_4 \leq (2+3C_3+C_4)K.
      \end{align*}      
     Substituting these two estimates into (\ref{eqn:contraction}) 
       completes the proof of the claim.
    \end{subproof}
  \begin{claim}[Claim 2]
   For $k^-\leq k\leq k^+$, we have $(2k-1)w/2 < \re \phi_k(z) < (2k+1)w/2$, provided
    $i$ was chosen sufficiently large.
  \end{claim}
  \begin{subproof}
    The hyperbolic distance in $V'$
     between $\omega_k$ and $\zeta_k'$ is uniformly bounded.
     The hyperbolic diameter of $Q+2\pi i m_k$ in $H$, and hence the hyperbolic
     diameter of $\phi_k(Q)$ in $V'$, is likewise uniformly bounded. Thus
     the hyperbolic distance between $\zeta_k'$ and $\phi_k(z)$ is uniformly bounded,
     independently of $k$ and $z\in Q$. On the other hand, as $k$ tends to infinity
     (under the assumption that 
       $K_i +w/2< w\cdot k + \rho_0 < 3K_i -w/2$
         for some $i$), we must have $\eps_k\to 0$
     by the same reasoning as in the 
     proof of Lemma \ref{lem:k(P)}. This implies that, for sufficiently large $k$, the
     set $\phi_k(Q)$ is contained in the $k$-th ``chamber'' of the tract $V'$, proving
     the claim. 
  \end{subproof}

 Consider the conformal iterated function system on $Q$ formed
   by the maps
     \[ \phi_k^{\ell}(z) := \phi_k(z) + 2\pi i \ell \]
    for $k^- \leq k \leq k^+$ and $|\ell| \leq (K-\pi)/2\pi$. By the preceding lemma, and
    choice of $\ell$, we have $\phi_k^{\ell}(Q)\subset Q$ for all $k$ and $\ell$, and 
    the images of $Q$ under these maps have pairwise disjoint closures. The
    number $N$ of functions in our IFS is
      \[ N = \left(2\left\lfloor \frac{K-\pi}{2\pi}\right\rfloor + 1\right)\cdot 
         (k^+ - k^- + 1 ) \geq
%%        \left(\frac{K}{\pi}-2\right) \cdot \left(\frac{2K}{w}-2\right)\geq
          \lambda_1\cdot K^2, \]
      where $\lambda_1$ is a suitable constant 
      (provided $i$, and hence $K$, is sufficiently large). 

  Let $X=X_i$ be the limit set of this iterated function system; i.e.\ $X$ is the unique
   compact set with 
    $X = \bigcup_{k,\ell} \phi_k^{\ell}(X)$.
   We have
     \[ \dimh(X_i) \geq \frac{\log N}{\inf_{k,\ell,z} \log |{(\phi_k^{\ell})^{-1}}'(z)|}
        \geq \frac{2\log K + \log \lambda_1}{\log K - \log \lambda_0}.  \]
    (Compare e.g.\ \cite[Lemma 2.10]{hypdim}.)
    So $\dimh(X_i)\to 2$ as $i$, and hence $K$, tends to infinity. 

  To conclude the proof, first note that $\exp(X_i)$ is invariant under $g$ by definition. 
    Indeed, let $z\in X_i$; say $z \in \phi_k^{\ell}(Q)$. Then 
      \begin{align*}
         g(\exp(z)) &= \exp(G_0(z-2\pi i \ell)) \\ &= \exp( (\phi_k)^{-1}(z-2\pi \ell)-2\pi i m_k)
              = \exp((\phi_k^{\ell})^{-1}(z)) \in \exp(X_i). \end{align*}
    So $\exp(X_i)$ is an invariant compact set for $g$. Every such set is a hyperbolic
     set for $g$ (since $g$ strictly expands
     the hyperbolic metric of $T$). This proves the claim for $g$, and the corresponding
     claim for the approximating function $f$ follows because the two functions are
     quasiconformally conjugate, and quasiconformal maps preserve sets of Hausdorff
     dimension two. Furthermore $g$ has finite order of growth by
     Lemma \ref{lem:sizeofG}, and the same holds for $f$. 
 \end{proof}
\begin{remark}[Remark 1]
  The key point in the proof at which the choice of the range $H$ comes into play is 
    Claim 1, which allows us to estimate the size of the pieces in the iterated function
    system from below. If our domain $H$ was, say, instead given by a ``parabola shape''
      \[ H = \{x + iy: x > -|y|^{\rho}\}, \]
    then the size of these pieces would shrink exponentially with $k$, and the
    proof breaks down completely.
\end{remark}
 \begin{remark}[Remark 2]
   The Ahlfors distortion theorem gives a precise value for the order of $g$ (and $f$)
     in terms of the shape of the tract $V$, and hence it is not difficult to see that,
     by varying the height $h$ of the tract, we can construct functions of any given
     order. Letting the height tend to zero (slowly) along the real axis, we can also
     construct functions of infinite order. 
 \end{remark}

To complete the proof of Theorem \ref{thm:hypdim}, it remains to show that---provided
   the $K_i$ were chosen appropriately---the Julia set of $f$ has positive area.
   This follows by letting the sequence $K_i$ grow extremely quickly, so as to 
   leave intermediate pieces where the tract has height $2\pi$, with the length of
   these pieces growing at least in an iterated exponential manner. 

  We can then apply \cite[Remark 3.1]{aspenbergbergweiler} to our
   function to see that the Julia set has positive area. That the hypotheses of this
   result are satisfied follows from \cite[Theorem 1.3]{aspenbergbergweiler}. We
   omit the details.

\section{Further applications}
\label{sec:examples}
 We now briefly comment on the construction of the other two
 counterexamples mentioned
  in the introduction.

 First we comment on Theorem \ref{thm:eremenkolowgrowth}. In
  \cite[Theorem 8.3]{strahlen}, a model function
  $\Psi:T\to\HH$ is constructed such that
  $z\mapsto e^{\Psi(z)}$ has the desired properties. Here the tract
  $T$ is constructed as $T=\exp(V)$, where $V$ is contained in a
  horizontal strip of height $2\pi$. Furthermore, the domain $V$ can
  be chosen to lie in any half plane
  $\{\re \zeta > R\}$; in particular for the choice $R=R_0=\log \rho_0$, where
   $\rho_0$ is as in Theorem \ref{thm:conjugacy}.

 By our theorems, we only need to show that, in this construction, we
  can replace the right half plane $\HH$ by the domain $H$ from
  Theorem \ref{thm:approx}. This can be seen in two ways. Either we 
  note that we can carry out the same construction in this setting also.
  Alternatively, let $\Psi$ be the model actually constructed in 
  \cite{strahlen}, and consider $\wt{\Psi} := \Psi\circ \Theta$, where
  $\Theta:\HH\to H$ is a conformal isomorphism. That this
  does not change the order of growth of the resulting function
  $e^{\wt{\Psi}}$ is easy to see by estimating the asymptotics of the
  map $\Theta$ using standard methods. 
  Thus it remains to show that the new map also has the property that
  the Julia set contains no unbounded path-connected components, and this
  follows exactly as in
  \cite[pp.\ 109--110: Proof of Theorem 1.1]{strahlen}. 

 \smallskip

 Now we turn to Theorem \ref{thm:devaney}. We indicate how to modify
  the proof in \cite{devaneyhairs}, where a model function
  $F:T\to \HH$ is constructed with the desired properties, to
  instead yield a model function $F:T\to H$, where $H$ is as in 
  Theorem \ref{thm:approx}. Here $T$ is a tract in logarithmic coordinates,
  so at the end of the construction, we will apply
  Theorem \ref{thm:approx} to the map
  $\Psi(e^z) := F(z)$. 

 The form of the tract $T$, as described in \cite[Section 6]{devaneyhairs} 
  is exactly the same; instead of choosing a conformal isomorphism
  $F:T\to \HH$ with $F(1)$ and
  $F(\infty)=\infty$, one takes instead a conformal isomorphism $F:T\to H$
  with the same properties. We need to show that
  \cite[Proposition 6.3]{devaneyhairs} also holds for this construction.

 Most of the proof again goes through verbatim, except for the following points,
  which previously used explicit formulas for the hyperbolic metric in
  $\HH$:
 \begin{itemize}
  \item Instead of the first displayed equation, we note that
     \[ \dist_H(1,F(u)) = \int_{1}^{F(u)} \rho_H(t)dt, \]
     and $1/t \geq \rho_H(t) \geq \frac{1}{2t}$ for all $t\in\R$, where
     the first inequality follows because $\HH\subset H$ and the
     second from the standard estimate on the hyperbolic metric in
     a simply-connected domain. Thus we see that
      \[ \dist_H(1,F(u))\leq \log(F(u)) \leq 2 \dist_H(1,F(u)). \]
     We again have $\dist_H(1,F(u))=\dist_T(1,u)$, and the proof proceeds
     as before.
  \item The estimate on $F(w_k)$ also uses the explicit formula 
    for the hyperbolic metric; this estimate is used in the choice
    of $\eps_k$ in the inductive construction.
    
   We could simply replace this estimate by that from
    Lemma \ref{lem:metricH}, but instead we
    make a more qualitative argument. The key point is
    that, for fixed $T>0$, the hyperbolic distance in $H$ between
    $R$ and the line $\{\re z = T\}$ tends to infinity as $R\to\infty$ by Lemma
    \ref{lem:metricH}.
              
   It follows that there is a function $\eta(R_0,r)$ with the following
    properties such that $R\geq R_0$ and
    $z\in H$ with $\dist_H(R,z)\leq r$, then
    $\re z > \eta(R_0,r)$. Furthermore, 
    for fixed $r$, we have $\lim_{R_0\to\infty} \eta(R_0,r)=\infty$. 

   Now, in the inductive definition, we again choose $\eps_k$ after
    $u_{k+1}$ has been chosen, but before $r_{k+1}$, in such a way that
       \[ \eta(b(r_k+1),12r_k) > u_{k+1} + \theta_k. \]
    (This is possible because $b(r_k+1)\to\infty$ as $\eps_k\to 0$.) 

   Since $\dist_H(F(w_k),F(r_k+1))\leq 12 r_k$ and
   $F(r_k+1)\geq b(r_k+1)$, we then again have
   $\re F(w_k)> u_{k+1}+\theta_k$, as desired. 
 \end{itemize}

\bibliographystyle{hamsplain}
\bibliography{/Latex/Biblio/biblio}

\end{document}